\documentclass[11pt]{article}

\usepackage{comment,url,algorithm,algorithmic,graphicx,subcaption,relsize}
\usepackage{amssymb,amsfonts,amsmath,amsthm,amscd,dsfont,mathrsfs,mathtools,nicefrac}
\usepackage{float,psfrag,epsfig,color,xcolor,url,hyperref}
\usepackage{epstopdf,bbm,mathtools,enumitem}

\usepackage[top=1in, bottom=1in, left=1in, right=1in]{geometry}

\def\balign#1\ealign{\begin{align}#1\end{align}}
\def\baligns#1\ealigns{\begin{align*}#1\end{align*}}
\def\balignat#1\ealign{\begin{alignat}#1\end{alignat}}
\def\balignats#1\ealigns{\begin{alignat*}#1\end{alignat*}}
\def\bitemize#1\eitemize{\begin{itemize}#1\end{itemize}}
\def\benumerate#1\eenumerate{\begin{enumerate}#1\end{enumerate}}

\newenvironment{talign*}
 {\csname align*\endcsname}
 {\endalign}
\newenvironment{talign}
 {\csname align\endcsname}
 {\endalign}

\def\balignst#1\ealignst{\begin{talign*}#1\end{talign*}}
\def\balignt#1\ealignt{\begin{talign}#1\end{talign}}

\let\originalleft\left
\let\originalright\right
\renewcommand{\left}{\mathopen{}\mathclose\bgroup\originalleft}
\renewcommand{\right}{\aftergroup\egroup\originalright}

\def\tinycitep*#1{{\tiny\citep*{#1}}}
\def\tinycitealt*#1{{\tiny\citealt*{#1}}}
\def\tinycite*#1{{\tiny\cite*{#1}}}
\def\smallcitep*#1{{\scriptsize\citep*{#1}}}
\def\smallcitealt*#1{{\scriptsize\citealt*{#1}}}
\def\smallcite*#1{{\scriptsize\cite*{#1}}}

\def\reals{\mathbb{R}} 

\def\<{\left\langle} 
\def\>{\right\rangle}

\def\bs{\backslash} 

\def\norm#1{\|{#1}\|} 
\newcommand{\inner}[1]{{\langle #1 \rangle}} 

\def\E{\mbb{E}} 

\def\P{\mbb{P}} 

\def\T{\top} 

\def\normal{{\sf N}}

\newcommand{\grad}{\nabla}

\providecommand{\arccos}{\mathop\mathrm{arccos}}

\providecommand{\diag}{\mathop\mathrm{diag}}
\providecommand{\tr}{\mathop\mathrm{tr}}

\providecommand{\minimize}{\mathop\mathrm{minimize}}
\providecommand{\maximize}{\mathop\mathrm{maximize}}

\ifdefined\nonewproofenvironments\else
\ifdefined\ispres\else
\newtheorem{theorem}{Theorem}
\newtheorem{lemma}[theorem]{Lemma}
\newtheorem{corollary}[theorem]{Corollary}
\newtheorem{definition}[theorem]{Definition}

\renewenvironment{proof}{\noindent\textbf{Proof}\hspace*{1em}}{\qed\\}
\newenvironment{proof-sketch}{\noindent\textbf{Proof Sketch}
  \hspace*{1em}}{\qed\bigskip\\}
\newenvironment{proof-idea}{\noindent\textbf{Proof Idea}
  \hspace*{1em}}{\qed\bigskip\\}
\newenvironment{proof-of-lemma}[1][{}]{\noindent\textbf{Proof of Lemma {#1}}
  \hspace*{1em}}{\qed\\}
\newenvironment{proof-of-theorem}[1][{}]{\noindent\textbf{Proof of Theorem {#1}}
  \hspace*{1em}}{\qed\\}
\newenvironment{proof-attempt}{\noindent\textbf{Proof Attempt}
  \hspace*{1em}}{\qed\bigskip\\}

\newenvironment{remark}{\noindent\textbf{Remark}
  \hspace*{1em}}{\smallskip}

\fi

\fi
\makeatletter
\@addtoreset{equation}{section}
\makeatother

\hypersetup{
  colorlinks,
  linkcolor={red!50!black},
  citecolor={blue!50!black},
  urlcolor={blue!80!black}
}

\newcommand{\eq}[1]{\begin{align}#1\end{align}}

\usepackage[utf8]{inputenc} 
\usepackage[T1]{fontenc}    
\usepackage{lmodern}
\usepackage{url}            
\usepackage{booktabs}       
\usepackage{amsfonts}       
\usepackage{nicefrac}       
\usepackage{microtype}      

\usepackage{comment,url,graphicx,subcaption,relsize}
\usepackage{amssymb,amsfonts,amsmath,amsthm,amscd,dsfont,mathrsfs,mathtools,nicefrac}
\usepackage{float,psfrag,epsfig,color,xcolor,url,hyperref}
\usepackage{epstopdf,bbm,mathtools,enumitem}

\usepackage{algorithm}

\hypersetup{
  colorlinks,
  linkcolor={red!50!black},
  citecolor={blue!50!black},
  urlcolor={blue!80!black}
}

\renewcommand\bs{\boldsymbol\sigma}
\newcommand\bsb{\bar{\boldsymbol\sigma}}
\newcommand\bu{\boldsymbol u}
\newcommand\br{\boldsymbol r}
\newcommand\bv{\boldsymbol v}
\def\bA{\boldsymbol A}
\def\bB{\boldsymbol B}
\def\bI{\boldsymbol I}
\def\bQ{\boldsymbol Q}
\def\bX{\boldsymbol X}
\def\bZ{\boldsymbol Z}
\def\bH{\boldsymbol H}
\def\bG{\boldsymbol G}
\def\bL{\boldsymbol \Lambda}
\def\bLe{\boldsymbol L}
\def\s{\sigma}
\def\Sym{\mathcal S}

\def\T{\top}

\def\<{\langle}
\def\>{\rangle}
\def\grad{\nabla}
\def\hess{\boldsymbol{\nabla}^2}

\def\normal{{\sf N}}
\def\reals{{\mathbb R}}

\def\maximize{\text{maximize}}
\def\minimize{\text{minimize}}

\def\Sym{{\mathrm{Sym}}}
\newcommand{\trace}[1]{{\rm Tr}\left( #1 \right)}
\newcommand{\Otilde}[1]{\widetilde{\mathcal{O}}\left(#1\right)}
\renewcommand{\O}[1]{\mathcal{O}\left(#1\right)}

\newcommand{\nn}{\nonumber}
\def\E{\mathbb{E}}
\def\S{\mathbb{S}}
\newcommand{\normf}[1]{\|#1\|_{\mathrm{F}}}
\def\M{\mathcal{M}}

\def\V{\mathcal{V}}
\def\P{\mathrm{P}}
\def\rgrad{{\mathrm{grad}}}
\def\rhess{{\mathrm{Hess}}}
\def\dist{\mathrm{dist}}
\def\diag{\mathrm{diag}}
\def\Diag{\mathrm{Diag}}
\def\dgrad{\mathrm{D \, grad}}
\def\so{\mathrm{O}}
\def\Exp{\mathrm{Exp}}
\newcommand{\ddiag}[1]{\mathrm{ddiag}\left(#1\right)}
\def\bsb{\bar{\bs}}

\def\gt{\tilde{g}}
\def\vt{\tilde{v}}
\def\Kbcm{K_{\mathrm{BCM}}}
\def\Khess{K_{\mathrm{H}}}

\title{Convergence Rate of Block-Coordinate Maximization Burer-Monteiro Method for Solving Large SDPs}

\author{
  Murat A. Erdogdu\thanks{University of Toronto, Department of Computer Science and Department of Statistical Sciences, \texttt{erdogdu@cs.toronto.edu}},\ \ \ Asuman Ozdaglar\thanks{Massachusetts Institute of Technology, Department of Electrical Engineering and Computer Science \texttt{\{asuman,parrilo,denizcan\}@mit.edu}},\ \ \ Pablo A. Parrilo\footnotemark[2],\ \ \ N.~Denizcan Vanli\footnotemark[2]
}

\begin{document}

\maketitle

\begin{abstract}
  Semidefinite programming (SDP) with diagonal constraints arise in many optimization problems, such as Max-Cut, community detection and group synchronization. Although SDPs can be solved to arbitrary precision in polynomial time, generic convex solvers do not scale well with the dimension of the problem. In order to address this issue, Burer and Monteiro \cite{burer2003nonlinear} proposed to reduce the dimension of the problem by appealing to a low-rank factorization and solve the subsequent non-convex problem instead. In this paper, we present coordinate ascent based methods to solve this non-convex problem with provable convergence guarantees. More specifically, we prove that the block-coordinate maximization algorithm applied to the non-convex Burer-Monteiro method globally converges to a first-order stationary point with a sublinear rate without any assumptions on the problem. We further show that this algorithm converges linearly around a local maximum provided that the objective function exhibits quadratic decay. We establish that this condition generically holds when the rank of the factorization is sufficiently large. Furthermore, incorporating Lanczos method to the block-coordinate maximization, we propose an algorithm that is guaranteed to return a solution that provides $1-\O{1/r}$ approximation to the original SDP without any assumptions, where $r$ is the rank of the factorization. This approximation ratio is known to be optimal (up to constants) under the unique games conjecture, and we can explicitly quantify the number of iterations to obtain such a solution.
\end{abstract}

\section{Introduction}

A variety of problems in statistical estimation and machine learning require solving a combinatorial optimization problem, which are often intractable \cite{boyd1996sdp}. Semidefinite programs (SDP) are commonly used as convex relaxations for these problems, providing efficient algorithms with approximate optimality \cite{parrilo2003sdp}.
A generic SDP in this framework can be written as
\eq{\label{eq:sdp-cvx}\tag{CVX}
  \maximize ~&~ \inner{\bA,\bX} \\
  \nonumber
  \text{subject to} ~&~ X_{ii} = 1, \text{ for } i \in [n],\\
  \nonumber
  & \ \bX \succeq 0,
}
where $\bA,\bX\in\mathrm{Sym}_n$ (real symmetric matrices of size $n \times n$) and $[n] = \{1,2,...,n\}$. This problem appears as a convex relaxation to the celebrated Max-Cut problem \cite{goemans1995improved}, graphical model inference \cite{erdogdu2017inference}, community detection problems \cite{boumal2016community}, and group synchronization \cite{mei2017solving}.

Although SDPs serve as reliable relaxations to many combinatorial problems, the resulting convex problem is still computationally challenging.
Interior point methods can solve SDPs to arbitrary accuracy in polynomial-time, but they do not scale well with the problem dimension $n$.
A popular approach to remedy these limitations is to introduce a low-rank factorization $\bX = \bs\bs^\T$, where $\bs \in \reals^{n \times r}$ with $r$ denoting the rank. This reformulation removes the positive semidefinite cone constraint in \eqref{eq:sdp-cvx} since $\bX = \bs\bs^\T$ is guaranteed to be a positive semidefinite matrix, and choosing $r \ll n$ provides computational efficiency as well as storage benefits. This method is often referred to as Burer-Monteiro approach \cite{burer2003nonlinear}. Denoting $i$-th row of $\bs$ by $\s_i$, i.e., $\bs = [\s_1,\s_2,...,\s_n]^\T$, the resulting non-convex problem can be written as follows
\eq{\label{eq:sdp-noncvx}\tag{Non-CVX}
  \maximize ~&~ \inner{\bA, \bs\bs^\T} \\ 
  \nonumber
  \text{subject to} ~&~ \norm{\s_i} = 1, \text{ for } i \in [n].
}
In the original Burer-Monteiro approach \cite{burer2003nonlinear}, the authors proposed to use an augmented Lagrangian method for a general form SDP. However, it has been recently observed that feasible methods (such as block-coordinate maximization \cite{javanmard2016phase,wang2017mixing}, Riemannian gradient \cite{javanmard2016phase,mei2017solving} and Riemannian trust-region methods \cite{absil2007trust,boumal2016non,journee2010lowrank}) provide empirically faster rates since feasibility can be efficiently guaranteed via projection onto the Cartesian product of spheres. Despite overwhelming empirical evidence \cite{javanmard2016phase,mei2017solving,wang2017mixing}, convergence properties of these feasible methods are not well-understood (except the Riemannian trust-region method, for which a sublinear convergence rate is shown in \cite{boumal2016non} and a local superlinear convergence is shown in \cite{absil2007trust} with no rate estimate). Among these methods, block-coordinate maximization and Riemannian gradient ascent are simpler to implement and have computational complexity of $\O{nr}$ and $\O{n^2r}$, respectively, whereas Riemannian trust-region requires to solve the trust-region subproblem at each iteration, which is usually solved iteratively using the Lanczos method in a few iterations, whose per iteration requires $\O{n^2r}$ arithmetic operations. Furthermore, block-coordinate maximization does not have any step size or tuning parameters, unlike Riemannian gradient ascent and Riemannian trust-region methods. Empirical studies further motivate the use of block-coordinate maximization by presenting superior performance compared to existing methods on large-scale problems, often with linear convergence \cite{javanmard2016phase}. In this paper, we provide the first local and global convergence rate guarantees for the block-coordinate maximization method (applied to Burer-Monteiro approach) in the literature, which are consistent with the empirical performance of the algorithm. Our contributions can be summarized as follows:

\begin{itemize}[leftmargin=1cm]

\item We establish the global sublinear convergence of the block-coordinate maximization algorithm applied to \eqref{eq:sdp-noncvx} without any assumptions on the cost matrix $\bA$.
  
\item We show that this algorithm enjoys a linear rate around a neighborhood of any local maximum when the objective function satisfies the quadratic decay assumption.

\item We establish that the quadratic decay condition that leads to local linear convergence generically holds when the rank of the factorization satisfies $r \geq \sqrt{2n}$.

\item Incorporating Lanczos methods into the block-coordinate maximization procedure, we propose an algorithm that returns an approximate second-order stationary point of \eqref{eq:sdp-noncvx}. By choosing the rank of the factorization sufficiently large and selecting the parameters of the algorithm according to the cost matrix $\bA$, we show that the solution returned by this algorithm is not only an approximate local maximum to \eqref{eq:sdp-noncvx}, but also provides $1-\O{1/r}$ approximation to \eqref{eq:sdp-cvx}. We highlight that this approximation ratio is optimal under the unique games conjecture.

\item We validate our theoretical results via numerical examples and compare the performance of the block-coordinate maximization algorithm with various manifold optimization methods to demonstrate its performance.
\end{itemize}

\subsection{Related Work}\label{ssec:related}
There are numerous papers that analyze the landscape of the solution space of \eqref{eq:sdp-noncvx}. In particular, it is known that \eqref{eq:sdp-cvx} admits an optimal solution of rank $r$ such that $r(r+1)/2 \leq n$ \cite{barvinok95problems,pataki1998rank}. Using this observation, it has been shown in \cite{burer2003nonlinear,burer2005local,journee2010lowrank} that when $r \geq\sqrt{2n}$, if $\bs$ is a rank deficient second-order stationary point of \eqref{eq:sdp-noncvx}, then $\bs$ is a global maximum for \eqref{eq:sdp-noncvx} and $\bX=\bs\bs^\T$ is a global maximum for \eqref{eq:sdp-cvx}. The recent paper \cite{boumal2018deterministic} showed that when $r \geq\sqrt{2n}$, for almost all $\bA$, every $\bs$ that is a first-order stationary point is rank deficient. For arbitrary rank $r$, it is shown that all local maxima are within a $n \norm{\bA}_2 / \sqrt{r}$ gap from the optimum of \eqref{eq:sdp-cvx} \cite{montanari2016grothendieck}, and any $\varepsilon$-approximate concave point is within a $\text{Rg}\eqref{eq:sdp-noncvx}/(r-1)+n\varepsilon/2$ gap from the optimum of \eqref{eq:sdp-cvx} \cite{mei2017solving}, where $\text{Rg}\eqref{eq:sdp-noncvx}$ is the range of the problem \eqref{eq:sdp-noncvx}, i.e., the difference between the maximum and the minimum values of the objective in \eqref{eq:sdp-noncvx}.

Javanmard  \emph{et al.} \cite{javanmard2016phase} showed that when applied to solve \eqref{eq:sdp-noncvx}, Riemannian gradient ascent and block-coordinate maximization methods provide excellent numerical results, yet no convergence guarantee is provided. Similar experimental results are also observed in \cite{wang2017mixing} for the block-coordinate maximization algorithm and in \cite{mei2017solving} for the Riemannian gradient ascent algorithm. Concurrent to this work, in \cite{wang2017mixing}, the authors analyzed the convergence of the deterministic block-coordinate maximization algorithm. In particular, they showed that the deterministic block-coordinate maximization algorithm is asymptotically convergent (see \cite[Theorem 3.2]{wang2017mixing}) and enjoys a local liner convergence with no explicit rate estimates (see \cite[Theorem 3.5]{wang2017mixing}). They also proved that the deterministic block-coordinate maximization approach converges to a local maximum generically under random initialization using the center-stable manifold theorem similar to \cite{lee16}. These results hold under the assumption that the iterates generated by the algorithm satisfy a certain condition that is seemingly impossible to verify without actually running the algorithm. To alleviate this issue, the authors suggested using a coordinate ascent method with a sufficiently small step size, for which the aforementioned convergence results hold without this precarious assumption. In \cite{boumal2016global}, the authors provided a global sublinear convergence rate for the Riemannian trust-region method for general non-convex problems and these results have been used in \cite{boumal2016non,mei2017solving} for the non-convex Burer-Monteiro approach. Augmented Lagrangian methods have been proposed to solve \eqref{eq:sdp-noncvx} as well \cite{burer2003nonlinear,burer2005local}, however these methods do not benefit from separability of the manifold constraints, and hence are usually slower \cite{boumal2018deterministic}.

There also exist methods that solve \eqref{eq:sdp-cvx} by exploiting its special structure \cite{arora05,garber,klein96,steurer}. In particular, \cite{klein96} reduces \eqref{eq:sdp-cvx} to a sequence of approximate eigenpair computations that is efficiently solved using the power method. In \cite{arora05,steurer}, matrix multiplicative weights algorithm is used to approximately solve \eqref{eq:sdp-cvx}, and these ideas are extended in \cite{garber} using sketching techniques \cite{tropp17}. However, these methods require constructing $\bX\in\mathrm{Sym}_n$ explicitly, which is prohibitive when $n$ goes beyond a few thousands, whereas the Burer-Monteiro approach we consider in this paper easily scales to very large instances as the low-rank factorization decreases the dimension of the problem from $\O{n^2}$ to $\O{nr}$ with $r\ll n$. For time complexity comparison between these methods that are based on Lagrangian relaxation and the Burer-Monteiro approach in this paper, we refer to Corollary \ref{cor:iterationcomp}.

\subsection{Notations and Preliminaries}\label{sec:notation}
Throughout the paper, matrices are denoted with a boldface font, and all vectors are column vectors. The superscripts are used to denote iteration counters, i.e., $\bs^k$ denotes the value of $\bs$ at iteration $k$. For a vector $g$, $\norm{g}$ represents its Euclidean norm. For matrices $\bA,\bB$,
we write $\<\bA, \bB\> = \mathrm{trace}(\bA \bB^\top)$ for the inner product
associated to the Frobenius norm $\normf{\bA} = \sqrt{\<\bA,\bA\>}$.
$A_{ij}$ represents the entry at the $i$-th row and $j$-th column of $\bA$, $A_i$ represents its $i$-th row as a column vector, and $\norm{\bA}_1 = \max_{1 \leq j \leq n} \sum_{i=1}^n |A_{ij}|$ represents its $1$-norm, and $\norm{\bA}_{1,1} = \sum_{i,j=1}^n |A_{ij}|$ represents its $L_{1,1}$-norm. For a function $h$, $\grad h$ and $\rgrad h$ represent its Euclidean and Riemannian gradients, respectively. Similarly, $\hess h$ and $\rhess h$ represent its Euclidean and Riemannian Hessians, respectively. We let $\S^{m-1}$ denote the unit sphere in $\reals^m$. For a vector $y$, $\Diag(y)$ represents the diagonal matrix whose $i$-th diagonal entry is $y_i$. Similarly for a matrix $\bA$, $\diag(\bA)$ represents the vector whose $i$-th entry is $A_{ii}$.

For brevity, we assume without loss of generality that $\bA$ is a symmetric matrix and $A_{ii} = 0$, for all $i \in [n]$ (the latter assumption is removed in Section \ref{sec:lanczos} to keep our presentation consistent with the existing works in the literature). Indeed, if $\bA$ is not symmetric, then we can replace $\bA$ by $(\bA+\bA^\T)/2$, which is a symmetric matrix, and the objective value \eqref{eq:sdp-noncvx} remains the same for all $\bs\in\reals^{n \times r}$ since $\bs\bs^\T$ is symmetric. Similarly, replacing the diagonal entries of $\bA$ by zeros decreases the objective value by the constant $\trace{\bA}$ for all feasible $\bs$, since the diagonal entries of $\bs\bs^\T$ are equal to $1$. 

The rest of the paper is organized as follows. In Section~\ref{sec:algorithm}, we present the algorithm and discuss its per iteration cost. In Section~\ref{sec:conv}, we prove the global sublinear convergence and local linear convergence of the algorithm with explicit rate estimates. In Section~\ref{sec:lanczos}, we introduce a second-order method based on block-coordinate maximization and Lanczos method that is guaranteed to return solutions with global optimality guarantees. We also provide a global sublinear convergence rate estimate for this algorithm. We perform numerical experiments to validate our theoretical results in Section~\ref{sec:experiments} and conclude the paper in Section~\ref{sec:conclusion}.

\section{Block-Coordinate Maximization Algorithm}\label{sec:algorithm}

In this section, we discuss the block-coordinate maximization (BCM) algorithm, its update rule and per iteration computational cost.
Throughout the paper, we let $f : \reals^{n\times r} \to \reals$ denote the objective function of \eqref{eq:sdp-noncvx}, i.e.,
\begin{equation*}
	f(\bs) = \inner{\bA,\bs\bs^\T}.
\end{equation*}
Given the current iterate $\bs^k$, the BCM algorithm chooses a row $i_k\in[n]$ of the matrix $\bs^k$ and maximizes the following objective
\begin{equation*}
	f(\bs^k) = \sum_{i=1}^n \inner{\s_i^k,g_i^k}, \quad \text{where} \quad g_i^k := \sum_{j \neq i} A_{ij} \, \s_j^k,
\end{equation*}
over the block $\s_{i_k}^k \in \S^{r-1}$. More formally, we can write the update rule of the algorithm as follows
\begin{align}\label{eq:updateRule}
  \sigma_{i_k}^{k+1} & = \arg\max_{\norm{\zeta}=1} f(\bs^k) \coloneqq f(\sigma_1^k, \dots, \sigma_{i_k-1}^k, \zeta, \sigma_{i_k+1}^k, \dots, \sigma_n^k), \nn\\
		& = \arg\max_{\norm{\zeta}=1} 2 \inner{\zeta, g_{i_k}^k} + \sum_{i \neq i_k} \sum_{j \neq i, i_k} A_{ij} \inner{\s_i^k,\s_j^k}, \nn\\
		& = \arg\max_{\norm{\zeta}=1} \inner{\zeta, g_{i_k}^k} = \frac{g_{i_k}^k}{\norm{g_{i_k}^k}},
\end{align}
with the convention that $\sigma_{i_k}^{k+1} = \sigma_{i_k}^k$ when $\norm{g_{i_k}^k}=0$. Blocks $\s_{i_k}^k$ that are updated at each iteration can be chosen through any deterministic or randomized rule, and in this paper we focus on three coordinate selection rules: 

\begin{itemize}[leftmargin=1cm]
	\item Uniform sampling: $i_k=i$ with probability $p_i=1/n$.
	\item Importance sampling: $i_k=i$ with probability $p_i = \norm{g_i^k} / \sum_{j=1}^n \norm{g_j^k}$.
	\item Greedy coordinate selection: $i_k=\arg\max_{i\in[n]} (\norm{g_i^k} - \inner{\s_i^k,g_i^k})$.
\end{itemize}

\begin{algorithm}[!t]
\begin{algorithmic}
\STATE Initialize $\bs^0 \in \reals^{n \times r}$ and calculate $g^0_i = \sum_{j \neq i} A_{ij} \s^0_j$, for all $i \in [n]$.
\FOR{$k=0,1,2,\dotsc$}
\STATE Choose block $i_k = i$ using one of the coordinate selection rules.
\STATE $\s^{k+1}_{i_k} \leftarrow g^k_{i_k} / \norm{g^k_{i_k}}$.
\STATE $g^{k+1}_i \leftarrow g^k_i - A_{ii_k} \s^k_{i_k} + A_{ii_k} \s^{k+1}_{i_k}$, for all $i \neq i_k$.
\ENDFOR
\end{algorithmic}
\caption{Block-Coordinate Maximization (BCM) \label{alg:bcm}}
\end{algorithm}

Per iteration computational cost of the BCM algorithm with uniform sampling is $\O{nr}$ as after $i_k$ is chosen uniformly at random, $g_{i_k}^k$ can be computed in $2(n-1)r$ floating point operations. On the other hand, the BCM algorithm with importance sampling and greedy coordinate selection requires all $\{\norm{g_i^k}\}_{i=1}^n$, which can be naively computed in $\O{n^2r}$ floating point operations per iteration. Instead, a smarter implementation is to keep both $\{\s_i^k\}_{i=1}^n$ and $\{g_i^k\}_{i=1}^n$'s in the memory (only the current iterates, not all the past ones) and update them as presented in Algorithm \ref{alg:bcm}, which can be done in $2(n-1)r$ floating point operations. Therefore, per iteration computational cost of the BCM method with all three coordinate selection rules is $\O{nr}$ for dense $\bA$ (i.e., when no structure is available on $\bA$). However, in many SDP applications (such as Max-Cut and graphical model inference), $\bA$ is induced by a graph and letting $d$ denote the maximum degree of the graph that induces $\bA$, the computational cost of the BCM algorithm becomes $\O{dr}$. In comparison, per iteration computational complexity of the Riemannian gradient ascent algorithm is $\O{n^2r}$, whereas the Riemannian trust-region algorithm runs a few iterations of a subroutine (e.g., power method) to solve the trust-region subproblem, whose per iteration cost is typically $\O{n^2r}$.

\section{Convergence Rate of BCM}\label{sec:conv}
In this section, we analyze the convergence rate of the BCM algorithm. As the feasible set of the problem in \eqref{eq:sdp-noncvx} defines a smooth manifold, we will use certain tools from manifold optimization throughout the paper, which are highlighted in the following subsection. We refer to
\cite[Section 5.4]{absil2007manifold} for a more detailed treatment of this topic.

\subsection{Riemannian Geometry of the Problem}\label{ssec:riemann}
We define the following submanifold of matrices $\reals^{n \times r}$ that corresponds to the Riemannian geometry induced by the constraints of the problem \eqref{eq:sdp-noncvx} in the Euclidean space:
\begin{equation*}
	\M_r := \left\{ \bs=(\s_1, \dots, \s_n)^\T \in \reals^{n \times r} : \norm{\s_i}=1, \ \forall i \in [n] \right\}.
\end{equation*}
This manifold represents the Cartesian product of $n$ unit spheres in $\reals^r$. For any given point $\bs\in\M_r$, its tangent space can be found by taking the differential of the equality constraints as follows
\begin{equation*}
	T_{\bs} \M_r := \left\{ \bu =(u_1, \dots, u_n)^\T \in \reals^{n \times r} : \inner{u_i,\s_i}=0, \ \forall i \in [n] \right\}.
\end{equation*}
The exponential map $\Exp_{\bs}: T_{\bs} \M_r \to \M_r$ on this manifold can be found as $\bs' = \Exp_{\bs}(\bu)$ such that
\begin{equation*}
	\s_i' = \s_i \cos(\norm{u_i}) + \frac{u_i}{\norm{u_i}} \sin(\norm{u_i}),
\end{equation*}
and the geodesic distance between $\bs$ and $\bs'$ is given by
\begin{equation*}
	\dist(\bs,\bs') = \normf{\bu}^2.
\end{equation*}
Using these definitions, the Riemannian gradient of $f(\bs)$ can be explicitly written as follows
\begin{equation*}
	\rgrad f(\bs) = 2 \left(\bA - \bL \right) \bs,
\end{equation*}
where $\bL = \Diag(\diag(\bA\bs\bs^\T))$. Furthermore, we have the following property for the Riemannian Hessian of $f(\bs)$ along the direction of a vector $\bu \in T_{\bs}\M_r$:
\begin{equation*}
	\inner{\bu,\rhess f(\bs)[\bu]} = 2 \inner{\bu,(\bA-\bL)\bu}.
\end{equation*}
The detailed derivations of these can be found in Appendix \ref{app:riemann}.

\subsection{Global Rate of Convergence}\label{ssec:sublinear}
In the following theorem, we consider the BCM algorithm with greedy coordinate selection and show that its functional ascent can be related to the norm of the Riemannian gradient of the function evaluated at the current iterate. By doing so, we show that the BCM algorithm returns a solution with arbitrarily small Riemannian gradient.

\begin{theorem}\label{thm:sublinear}
Let $f^* = \max_{\norm{\s_i}=1,\forall i\in[n]} f(\bs)$. Then, for any $K\geq1$, BCM with greedy coordinate selection yields the following guarantee
\begin{equation}\label{eq:sublinearRate3}
	\min_{k\in[K-1]} \normf{\rgrad f(\bs^k)}^2 \leq \frac{2n \norm{\bA}_1 (f^* - f(\bs^0))}{K}.
\end{equation}
\end{theorem}

Using a similar approach to Theorem \ref{thm:sublinear}, we show in the following corollary that the BCM algorithm with uniform and importance sampling attains a similar sublinear convergence rate in expectation.

\begin{corollary}\label{cor:sublinear}
Let $f^* = \max_{\norm{\s_i}=1,\forall i\in[n]} f(\bs)$. Then, for any $K\geq1$, randomized BCM yields the following guarantee
\begin{equation}\label{eq:sublinearRate2}
	\min_{k\in[K-1]} \E \normf{\rgrad f(\bs^k)}^2 \leq \frac{2L (f^* - f(\bs^0))}{K},
\end{equation}
where
\begin{equation}\label{eq:sublinearRate}
	L = 
	\begin{cases}
	n \norm{\bA}_1, & \text{for uniform sampling,} \\
	\norm{\bA}_{1,1}, & \text{for importance sampling.}
	\end{cases}
\end{equation}
\end{corollary}

We can observe from \eqref{eq:sublinearRate3} and \eqref{eq:sublinearRate2} that BCM with uniform sampling attains the same sublinear rate as BCM with greedy coordinate selection in expectation as they both require at most $\left\lceil \big( 2n \norm{\bA}_1 (f^* - f(\bs^0)) \big) / \epsilon \right\rceil$ iterations to return a solution $\bs$ satisfying $\normf{\rgrad f(\bs)}^2 \leq\epsilon$. On the other hand, comparing the rate guarantees in \eqref{eq:sublinearRate2} and \eqref{eq:sublinearRate}, we see that BCM with importance sampling enjoys a tighter convergence rate compared to BCM with uniform sampling, since $\norm{\bA}_{1,1} \leq n \norm{\bA}_1$ for all $\bA\in\reals^{n \times n}$.

\subsection{Local Rate of Convergence}\label{ssec:linear}
Although the BCM algorithm enjoys the sublinear convergence rates presented in Section \ref{ssec:sublinear}, it is numerically observed that the rate of convergence is linear when $\bs^k$ is close to a local maximum \cite{javanmard2016phase,wang2017mixing}. In this section, we investigate this behavior and prove that indeed BCM attains a linear convergence rate around a local maximum under the quadratic decay condition on the objective function, which is classically defined as follows \cite{anitescu00,bonnans95}: Consider the unconstrained maximization problem: $\max_x \varphi(x)$, and let $\Omega_{\bar{x}}$ denote the set of local maximizers with objective value $\varphi(\bar{x})$. Then, the quadratic decay condition is said to be satisfied at $\bar{x}$ for $\varphi$, if there exists constants $\mu,\delta>0$ such that $\varphi(x) \leq \varphi(\bar{x}) - \mu \, \mathrm{dist}^2(x,\Omega_{\bar{x}})$, for all $x$ such that $\norm{x-\bar{x}} \leq \delta$, where $\dist$ measures the distance of the point $x$ to the set $\Omega_{\bar{x}}$.

For the constrained optimization problem that we are considering in \eqref{eq:sdp-noncvx}, this definition needs to be slightly reworked. In particular, let $\bs$ be a local maximum of \eqref{eq:sdp-noncvx} and consider the Taylor expansion of $\Exp_{\bs}(\bu)$ around $\bs$:
\begin{equation*}
	f(\Exp_{\bs}(\bu)) = f(\bs) + \frac{1}{2} \inner{\bu,\rhess f(\bs)[\bu]} + \O{\normf{\bu}^3},
\end{equation*}
where the first-order term is zero as $\bs$ is a local maximum. Then, for a sufficiently small neighborhood of $\bs$, the quadratic decay condition is satisfied if and only if there exists a constant $\mu>0$ such that $\inner{\bu,\rhess f(\bs)[\bu]} \leq - \mu \, \mathrm{dist}^2(\Exp_{\bs}(\bu),\Omega_{\bs})$, for all $\Exp_{\bs}(\bu)$ sufficiently close to $\bs$, where $\Omega_{\bs}$ is the set on which $f$ has constant value $f(\bs)$. Assume for the sake of simplicity that $\bs$ is a strict local maximum, i.e., $\Omega_{\bs} = \{\bs\}$. Then, the distance between $\Exp_{\bs}(\bu)$ and $\bs$ can be found as the norm of the tangent vector that connects these two points via the geodesic curve, i.e., $\mathrm{dist}(\Exp_{\bs}(\bu),\bs) = \normf{\bu}$. Therefore, the quadratic decay condition is satisfied if and only if there exists a constant $\mu>0$ such that $\inner{\bu,\rhess f(\bs)[\bu]} \leq - \mu \normf{\bu}^2$ for all $\bu \in T_{\bs}\M_r$, where we note that the condition that $\Exp_{\bs}(\bu)$ is sufficiently close to $\bs$ is dropped considering the limit as $\bu \to {\boldsymbol 0}$.

Unfortunately, no local maximum is a strict local maximum for the problem \eqref{eq:sdp-noncvx}. To observe this, let $\so(r)=\{\bQ\in\reals^{r \times r} : \bQ^\T \bQ = \bQ\bQ^\T = \bI\}$ denote the orthogonal group in dimension $r$. Then, it can be observed that $f(\bs \bQ) = \inner{\bA,\bs\,  \bQ \bQ^\T \bs^\T} = \inner{\bA,\bs \bs^\T} = f(\bs)$, for any $\bQ\in\so(r)$. Therefore, in order to measure the distance between $\Exp_{\bs}(\bu)$ and $\Omega_{\bs}$, we define the following equivalence relation $\sim$:
\begin{equation}\label{eq:equivalence}
	\bs \sim \bs' \quad \Longleftrightarrow \quad \exists \bQ \in O(r): \bs = \bs' \bQ.
\end{equation}
This equivalence relation induces a quotient space denoted by $\M_r/\sim$ and we let $[\bs]$ denote the equivalence class of a given matrix $\bs\in\M_r$. According to this definition, $f$ has constant value of $f(\bs)$ on the set $[\bs]$, i.e., $\Omega_{\bs} = [\bs]$. We let $\V_{\bs} \subset T_{\bs}\M_r$ denote the tangent space to the equivalence class $[\bs]$, which can be found as $\V_{\bs} = \{\bs \bB: \bB\in\reals^{r \times r} \text{ and } \bB^\T = -\bB\}$.\footnote{Note that the dimension of $\V_{\bs}$ depends on the rank of $\bs$, and hence the quotient space is not a manifold.} Therefore, $\mathrm{dist}(\Exp_{\bs}(\bu), [\bs]) = \normf{\bu}$ if the closest point to $\Exp_{\bs}(\bu)$ in $[\bs]$ is $\bs$, or equivalently $\mathrm{dist}(\Exp_{\bs}(\bu), [\bs]) = \normf{\bu}$ if $u \in T_{\bs}\M_r \setminus \V_{\bs}$. Consequently, we say that quadratic decay is satisfied at $\bs$ for $f$ if $\rhess f(\bs)$ is negative definite on the orthogonal complement of $\V_{\bs}$ in $T_{\bs}\M_r$. The formal statement of this definition is as follows.


\begin{definition}[Quadratic Decay]\label{asmp}
  Let $\bs$ be a local maximum of \eqref{eq:sdp-noncvx}. Quadratic decay condition is said to be satisfied at $\bs$ for $f$ if there exists a constant $\mu>0$ such that
  \eq{
    \inner{\bu,\rhess f(\bs)[\bu]} \leq -\mu \normf{\bu}^2,\  \text{ for all } \
    \bu\in T_{\bs}\M_r \setminus \V_{\bs},
  }
  where $\V_{\bs}$ is the tangent space to the equivalence class $[\bs]$.
\end{definition}

In the following theorem, we present the linear convergence rate of the BCM algorithm under the quadratic decay condition. We defer the validity of this condition to Section \ref{ssec:asmp} where we show that quadratic decay generically (over the set of matrices $A$) holds for $f$ when $r$ is sufficiently large.

\begin{theorem}\label{thm:linearRate}
  Let $\bsb$ be a limit point of the BCM algorithm and assume that $\bsb$ is a local maximum that satisfies the quadratic decay condition. If $\bs^0$ is sufficiently close to the equivalent class $[\bsb]$, then the iterates generated by the BCM algorithm with greedy coordinate selection enjoy the following linear convergence rate
\begin{equation}\label{eq:linearRate}
	f(\bsb) - f(\bs^{k+1}) \leq \left( 1 - \frac{\mu}{4n^2\norm{\bA}_1} \right) \left( f(\bsb) - f(\bs^k) \right).
\end{equation}
\end{theorem}

The linear convergence rate of the BCM algorithm with greedy coordinate selection in Theorem \ref{thm:linearRate} can be extended for importance sampling and uniform sampling as we highlight in the following.

\begin{corollary}\label{cor:linearRate}
Let the conditions in Theorem~\ref{thm:linearRate} hold. Then, the iterates generated by the BCM algorithm enjoys the local linear convergence rate
\begin{equation}\label{eq:linearRate}
	f(\bsb) - \E f(\bs^k) \leq \left( 1 - \rho \right)^k \left( f(\bsb) - f(\bs^0) \right),
\end{equation}
where $\rho = \frac{\mu}{4n\norm{\bA}_{1,1}}$ for importance sampling and $\rho = \frac{\mu}{4n^2\norm{\bA}_1}$ for uniform sampling.
\end{corollary}

\subsection{Quadratic Decay Condition Holds Generically}\label{ssec:asmp}
In this section, we consider the quadratic decay condition, which is a condition on \eqref{eq:sdp-noncvx}, and relate it to a condition on the original problem in \eqref{eq:sdp-cvx}. In particular, we characterize sufficient conditions on \eqref{eq:sdp-cvx} for quadratic decay to hold. We first provide some background on semidefinite programming (see for example \cite{alizadeh97}, for a more detailed treatment of this topic). Consider the SDP in \eqref{eq:sdp-cvx}:
\begin{align*}
  \maximize ~&~ \inner{\bA, \bX} \\
  \text{subject to} ~&~ X_{ii} = 1, \text{ for } i \in [n],\\
  & \ \bX \succeq 0,
\end{align*}
and its dual:
\begin{align*}
  \minimize ~&~ \inner{1, y} \\
  \text{subject to} ~&~ \bZ = \Diag(y) - \bA,\\
  & \ \bZ \succeq 0,
\end{align*}
where $1$ is the vector of ones of appropriate size. Let $\bX^*$ and $(y^*,\bZ^*)$ denote the primal and dual optimal solutions, respectively, and let $r^*$ denote the rank of $\bX^*$. Then, there exists a $\bQ \in O(n)$ such that
\begin{align*}
	\bX^* & = \bQ \, \Diag(\lambda_1,\dots,\lambda_{r^*},0,\dots,0) \, \bQ^\T, \\
	\bZ^* & = \bQ \, \Diag(0,\dots,0,\omega_{r^*+1},\dots,\omega_n) \, \bQ^\T.
\end{align*}
We say that {\em strict complementarity} holds if $\lambda_i>0$ for $i=1,\dots,r^*$ and $\omega_j>0$ for $j=r^*+1,\dots,n$. Furthermore, let $\bQ_1 \in \reals^{n \times r^*}$ and $\bQ_2 \in \reals^{n \times (n-r^*)}$ respectively denote the first $r^*$ columns and the last $n-r^*$ columns of $\bQ$ and let $q_i$ denote the $i$th row of $\bQ_1$, i.e., $\bQ_1 = [q_1, q_2, \dots, q_n]^\T$. Then, $(y^*,\bZ^*)$ is {\em dual nondegenerate} if and only if $\{q_1 q_1^\T, \dots, q_n q_n^\T\}$ spans $\Sym_{r^*}$, i.e., the set of real symmetric $r^* \times r^*$ matrices \cite[Theorem 3]{alizadeh97}. If $\bX^*$ and $(y^*,\bZ^*)$ satisfy strict complementarity and $(y^*,\bZ^*)$ is dual nondegenerate, then $\bX^*$ is the unique optimal primal solution (see  \cite[Theorem 4]{alizadeh97}). Strict complementarity and dual nondegeneracy are known to hold generically (over the set of possible cost matrices $\bA \in \reals^{n \times n}$, i.e., they fail to hold only on a subset of measure zero of $\reals^{n \times n}$) as proven in \cite[Lemma 2]{alizadeh97}. Using these definitions, we show in the next theorem that strict complementarity and dual nondegeneracy are sufficient for quadratic decay to hold at the maximizer of \eqref{eq:sdp-noncvx}.

\begin{theorem}\label{lem:asmpGeneral}
Suppose that $\bX^*$ and $(y^*,\bZ^*)$ are respectively primal and dual optimal solutions satisfying strict complementarity and dual nondegeneracy. If $r \geq \mathrm{rank}(\bX^*)$, then quadratic decay is satisfied for $f$ at all $\bs$ such that $\bs\bs^\T = \bX^*$.
\end{theorem}

This theorem states that quadratic decay holds for all global maxima of \eqref{eq:sdp-noncvx} provided that the rank of the factorization is large enough so that the global maximum values of \eqref{eq:sdp-cvx} and \eqref{eq:sdp-noncvx} are equal to one another. For this case, the set of all global maxima is an equivalence class corresponding to a solution since strict complementarity and dual nondegeneracy imply that the primal solution of \eqref{eq:sdp-cvx} is unique. On top of this, when $r \geq \sqrt{2n}$, it is known that (see \cite[Theorem 2]{boumal2016non}) any local maximum is global generically (i.e., for almost all cost matrices $\bA$). As strict complementarity and dual nondegeneracy also hold generically for \eqref{eq:sdp-cvx}, then consequently, when $r \geq \sqrt{2n}$, quadratic decay holds for all local maxima generically as we highlight in the following corollary.

\begin{corollary}\label{cor:asmp}
If $r \geq \sqrt{2n}$, then quadratic decay holds for all local maxima generically.
\end{corollary}

\section{Approximately Achieving the Maximum Value of \eqref{eq:sdp-cvx}}\label{sec:lanczos}
Our results in Section \ref{sec:conv} show that the BCM algorithm converges with a sublinear rate to a first-order stationary solution and with a linear rate to a local maximum when initialized sufficiently close to it. In this section, we incorporate a second-order oracle to BCM in order to obtain an algorithm, which we refer to as BCM2, that returns an approximate second-order stationary point. More specifically, at the current iteration of the algorithm, if the norm of the gradient is large, we take a BCM step. Otherwise, we run a subroutine (e.g., Lanczos method) to find the leading eigenvector of the Hessian. The main motivation for designing such an algorithm is that the approximate second-order stationary solutions provide $\O{1/r}$ approximation to \eqref{eq:sdp-cvx}. In particular, call $\bs$ an {\em $\varepsilon$-approximate concave point} if $\inner{\bu,\rhess f(\bs)[\bu]} \leq \varepsilon \inner{\bu,\bu}$, for all $\bu \in T_{\bs}\M_r$. Then, the following theorem provides an approximation ratio between the approximate concave points of \eqref{eq:sdp-noncvx} and the maximum value of \eqref{eq:sdp-cvx}.

\begin{theorem}[{\cite[Theorem 1]{mei2017solving}}]\label{thm:andrea}
Let $\bs\in\M_r$ be an $\varepsilon$-approximate concave point. Then, for any positive semidefinite $\bA$, the following approximation ratio holds:
\begin{equation}
	f(\bs) \geq \left( 1-\frac{1}{r-1} \right) \mathrm{SDP}(\bA) - \frac{n}{2} \varepsilon,
\end{equation}
where $\mathrm{SDP}(\bA)$ is the maximum value of \eqref{eq:sdp-cvx}.
\end{theorem}

This approximation ratio follows due to a generalization of the randomized rounding approach (most famously presented by \cite{goemans1995improved}) applied to an $\varepsilon$-approximate concave point. In fact, it can be shown that it is not possible to find a better approximation ratio (in terms of the dependence on the rank of the factorization $r$) for all problems $\bA$. This result is highlighted in the following theorem.

\begin{theorem}[{\cite[Theorems 1 \& 3]{briet10}}]
Let $\mathrm{SDP}(\bA)$ be the maximum value of \eqref{eq:sdp-cvx} and $\mathrm{SDP}_r(\bA)$ be the maximum value of \eqref{eq:sdp-noncvx}. Then, for all positive semidefinite matrices $\bA$, the following approximation ratio holds:
\begin{equation}
	1 \geq \frac{\mathrm{SDP}_r(\bA)}{\mathrm{SDP}(\bA)} \geq \gamma(r) = \frac{2}{r} \left( \frac{\Gamma((r+1)/2)}{\Gamma(r/2)} \right)^2 = 1 - \Theta(1/r),
\end{equation}
where $\Gamma(z) = \int_0^\infty x^{z-1} e^{-x} dx$ is the Gamma function. Furthermore, under the unique games conjecture, there is no polynomial-time algorithm that approximates $\mathrm{SDP}_r(\bA)$ with an approximation ratio greater than $\gamma(r)+\varepsilon$ for any $\varepsilon>0$.
\end{theorem}

These results provide motivation to design algorithms with second-order guarantees to solve \eqref{eq:sdp-noncvx} and for this reason, we propose the BCM2 algorithm (see Algorithm \ref{alg:bcm-soo}), which can be described as follows: When the Frobenius norm of the Riemannian gradient is at least as large as $\normf{\rgrad f(\bs^k)}^2 > \epsilon^3/(1350\norm{\bA}_1)$, we use BCM to update the current solution. Otherwise, we assume that there is a second-order oracle that returns an update direction $\bu^k\in T_{\bs}\M_r$ such that $\inner{\bu^k, \rhess f(\bs^k)[\bu^k]} \geq \lambda_{\max}(\rhess f(\bs^k))/2$, $\inner{\bu^k, \rgrad f(\bs^k)} \geq 0$, and $\normf{\bu^k}=1$. Notice that finding a tangent vector $\bu^k$ that satisfy $\inner{\bu^k, \rhess f(\bs^k)[\bu^k]} \geq \lambda_{\max}(\rhess f(\bs^k))/2$ and $\normf{\bu^k}=1$ is an eigenpair problem and can be solved efficiently using the Lanczos method. The condition $\inner{\bu^k, \rgrad f(\bs^k)} \geq 0$, on the other hand, can always be satisfied by switching the sign of $\bu^k$. It is a straightforward exercise to explicitly construct such a vector and it can be found in \cite[Lemma 11]{boumal2016global}.

\begin{algorithm}[t]
\begin{algorithmic}[1]
\STATE Initialize $\bs^0 \in \reals^{n \times r}$ and calculate $g^0_i = \sum_{j \neq i} A_{ij} \s^0_j$, for all $i \in [n]$.
\FOR{$k=0,1,2,\dotsc$}
\STATE Compute $\normf{\rgrad f(\bs^k)}^2 = 2\sum_{i=1}^n (\norm{g_i^k}^2-\inner{\s_i^k,g_i^k}^2)$.
\IF{$\normf{\rgrad f(\bs^k)}^2 > \varepsilon^3/(1350\norm{\bA}_1)$} 
\STATE $i_k \leftarrow \arg\max_{i\in[n]} (\norm{g_i^k}-\inner{\s_i^k,g_i^k})$
\STATE $\s^{k+1}_{i_k} \leftarrow g^k_{i_k} / \norm{g^k_{i_k}}$.
\STATE $g^{k+1}_i \leftarrow g^k_i - A_{ii_k} \s^k_{i_k} + A_{ii_k} \s^{k+1}_{i_k}$, for all $i \neq i_k$.
\ELSE
\STATE Find a direction $\bu^k\in T_{\bs^k}\M_r$ such that $\inner{\bu^k, \rhess f(\bs^k)[\bu^k]} \geq \lambda_{\max}(\rhess f(\bs^k))/2$, $\inner{\bu^k, \rgrad f(\bs^k)} \geq 0$, and $\normf{\bu^k}=1$.
\STATE $\s^{k+1}_i \leftarrow \s^k_i \cos(\norm{u_i^k}t) + \frac{u_i^k}{\norm{u_i^k}} \sin(\norm{u_i^k}t)$, for all $i\in[n]$, where $t=\varepsilon/(15\norm{\bA}_1)$.
\STATE $g^{k+1}_i \leftarrow \sum_{j \neq i} A_{ij} \s^{k+1}_j$, for all $i \in [n]$.
\ENDIF
\ENDFOR
\end{algorithmic}
\caption{BCM2 \label{alg:bcm-soo}}
\end{algorithm}

We first analyze the convergence of Algorithm \ref{alg:bcm-soo} in Theorem \ref{thm:iterationcomp}, where we assume that we have access to a subroutine that solves the eigenpair problem to the desired accuracy. We then implement the subroutine using the Lanczos algorithm (presented in Algorithm \ref{alg:lanczos}) and present its convergence in Theorem \ref{thm:iterationcomp_whp}. In particular, we have the following theorem for the former case.

\begin{theorem}\label{thm:iterationcomp}
Consider Algorithm \ref{alg:bcm-soo}, where BCM is used at iteration $k$ if $\normf{\rgrad f(\bs^k)}^2 \geq \varepsilon^3 / (1350\norm{\bA}_1)$ and a second-order step (see lines 9-11 of Algorithm \ref{alg:bcm-soo}) is taken otherwise. Let $\Kbcm$ denote the number of BCM epochs made and let $\Khess$ denote the number of second-order oracle iterations made such that $K=n\Kbcm+\Khess$. Then, as soon as
\begin{equation}\label{eq:itercomp1}
	\Kbcm+\Khess = \left\lceil \frac{675 n \norm{\bA}_1^2}{\varepsilon^2} \right\rceil,
\end{equation}
Algorithm \ref{alg:bcm-soo} is guaranteed to return a solution $\bs^K$ that satisfies
\begin{equation}\label{eq:itercomp2}
	f(\bs^K) \geq \left( 1-\frac{1}{r-1} \right) \mathrm{SDP}(\bA) - \frac{n}{2} \varepsilon,
\end{equation}
where $\mathrm{SDP}(\bA)$ is the maximum value of \eqref{eq:sdp-cvx}.
\end{theorem}

In Theorem \ref{thm:iterationcomp}, $\Kbcm+\Khess$ represents the total number of epochs to guarantee \eqref{eq:itercomp2}, whereas the iteration counter of the algorithm is given in terms of $K=n\Kbcm+\Khess$. This is due to the fact that, at each iteration of the BCM algorithm, a single row of $\bs$ is updated and consequently $n$ iterations of the BCM algorithm add up to an epoch. On the other hand, at each iteration of the second-order step, all entries of $\bs$ are updated, and hence each second-order iteration is an epoch. In terms of the computational cost, an iteration of BCM requires $\O{nr}$ operations and consequently an epoch of BCM requires $\O{n^2r}$ operations, whereas the second-order direction of update is typically found approximately via a few iterations of the power method or the Lanczos method (see Theorem \ref{thm:iterationcomp_whp} for a more rigorous treatment of this statement), which require $\O{n^2r}$ operations. Therefore, an epoch of Algorithm \ref{alg:bcm-soo} typically has a computational complexity of $\O{n^2r}$. Furthermore, by Theorem \ref{thm:iterationcomp}, we observe that in at most $\O{n\norm{\bA}_1^2/\varepsilon^2}$ epochs, Algorithm \ref{alg:bcm-soo} returns a solution that achieves the optimal approximation ratio up to an accuracy of $\O{n\varepsilon}$. In particular, picking $\varepsilon = 2 \, \mathrm{SDP}(\bA) / (n(r-1))$, we obtain the following corollary.

\begin{corollary}\label{cor:iterationcomp}
Consider the setup of Theorem \ref{thm:iterationcomp} and set $\varepsilon = 2 \, \mathrm{SDP}(\bA) / (n(r-1))$. Then, as soon as
\begin{equation}\label{eq:K}
	K = \left\lceil \frac{675 n^3 (r-1)^2 \norm{\bA}_1^2}{4 (\mathrm{SDP}(\bA))^2} \right\rceil,
\end{equation}
Algorithm \ref{alg:bcm-soo} is guaranteed to return a solution $\bs^K$ that satisfies
\begin{equation*}
	f(\bs^K) \geq \left( 1-\frac{2}{r-1} \right) \mathrm{SDP}(\bA).
\end{equation*}
\end{corollary}

\begin{remark}
In order to understand the total running time of BCM2, consider the following example. Let $\bA$ be the adjacency matrix of a random Erdos-R\'{e}nyi graph on $n$ nodes and $\lfloor cn \rfloor$ edges. The size of the maximum cut in this graph normalized by the number of nodes can be bounded between $[c/2+0.4\sqrt{c},c/2+0.6\sqrt{c}]$ with high probability as $n$ increases, for all sufficiently large $c$ \cite{gamarnik2014maxcut}. Since the maximum value of \eqref{eq:sdp-cvx} is within $0.878$ of the maximum cut \cite{goemans1995improved}, we can then conclude that $\mathrm{SDP}(\bA)/n = \O{c}$ with high probability. We can also observe that for this graph, the degree of a node approximately follows a Poisson distribution with mean $2c$, which can be approximated by a normal distribution with mean $2c$ and variance $\sqrt{2c}$, for large $c$ \cite{gamarnik2014maxcut}. Then, we have $\norm{\bA}_1 = \O{c \log n}$ with high probability. Therefore, for this problem, Corollary \ref{cor:iterationcomp} states that in $\Otilde{nr^2}$ iterations (where tilde is used to hide the logarithmic dependences), Algorithm \ref{alg:bcm-soo} returns a $\O{1/r}$-optimal solution with high probability. Per iteration computational cost of the  algorithm is $\O{nrc}$, which results in a total running time of $\Otilde{n^2r^3c}$. In comparison, Klein-Lu method (see \cite[Lemma 4]{klein96}) requires $\Otilde{n^2r^3c}$ running time and the matrix multiplicative weights method (see \cite[Theorem 3]{arora05}) requires $\Otilde{n^2r^{3.5}/c}$ running time to return a $1/r$-optimal solution.
\end{remark}

In the description of Algorithm \ref{alg:bcm-soo} (see line 9), we assumed that we have access to a vector in the tangent space of the current iterate, which satisfies certain second-order conditions. In Algorithm \ref{alg:lanczos}, we describe an efficient subroutine to find this desired tangent vector based on the Lanczos method. In particular, the Lanczos method returns a tridiagonal real symmetric matrix whose diagonal entries are $\{\alpha_\ell\}_{\ell\geq1}$ and off-diagonal entries are $\{\beta_\ell\}_{\ell\geq2}$, where $\ell$ denotes the iteration counter in Algorithm \ref{alg:lanczos}. The entire spectrum of such a symmetric tridiagonal matrix can be efficiently computed in almost linear time in the dimension of the matrix \cite{coakley2013}. Consequently, letting $y$ denote the leading eigenvector of this tridiagonal matrix, we can construct the desired tangent vector in Algorithm \ref{alg:bcm-soo} as $\bu^k = \sum_{\ell\geq1} y_\ell \bu_\ell$. It is well-known that after $n(r-1)$ iterations, the Lanczos method constructs the leading eigenvector exactly (since order-$n(r-1)$ Krylov subspace spans the entire tangent space). Furthermore, it is also possible to analyze the performance of the Lanczos method with early termination  \cite{kuczynski1992}. Building on these ideas, we characterize the quality of the solution returned by Algorithms \ref{alg:bcm-soo}+\ref{alg:lanczos} in the following theorem.

\begin{algorithm}[t]
\begin{algorithmic}[1]
\STATE Given $\bs$, define $H[\bu] = \rhess f(\bs)[\bu] + 4\norm{A}_1 \bu$. Initialize $\bu_1 \in T_{\bs}\M_r$ such that $\normf{\bu_1}=1$. Let $\alpha_1 = \inner{\bu_1, H[\bu_1]}$ and $\br_1 = H[\bu_1] - \alpha_1 \bu_1$.
\FOR{$\ell \geq 2$}
\STATE $\beta_\ell = \normf{\br_{\ell-1}}$
\STATE $\bu_\ell = \br_{\ell-1} / \beta_\ell$ (If $\beta_\ell=0$, pick $\bu_\ell \perp \mathrm{span}(\bu_1,\dots, \bu_{\ell-1})$ arbitrarily)
\STATE $\alpha_\ell = \inner{\bu_\ell, H[\bu_\ell]}$
\STATE $\br_\ell = H[\bu_\ell] - \alpha_\ell \bu_\ell - \beta_\ell \bu_{\ell-1}$
\ENDFOR
\end{algorithmic}
\caption{Lanczos Method \label{alg:lanczos}}
\end{algorithm}

\begin{theorem}\label{thm:iterationcomp_whp}
Suppose in Algorithm \ref{alg:lanczos}, we initialize $\bu_1$ uniformly at random over $T_{\bs}\M_r$. Let
\begin{equation*}
	\ell^* = \left\lceil \left( \frac{1}{2} + 2 \sqrt{\frac{\norm{\bA}_1}{\varepsilon}} \right) \log\left( \frac{\left\lceil \frac{675 n \norm{\bA}_1^2}{\varepsilon^2} \right\rceil 1.648 \sqrt{n(r-1)}}{\delta} \right) \right\rceil,
\end{equation*}
and consider that Algorithm \ref{alg:lanczos} is run for $\min(\ell^*,n(r-1))$ iterations at each call from Algorithm \ref{alg:bcm-soo}. Then, after $K$ iterations (defined as in \eqref{eq:K}), Algorithm \ref{alg:bcm-soo} returns a solution $\s^K$ that satisfies
\begin{equation*}
	f(\bs^K) \geq \left( 1-\frac{1}{r-1} \right) \mathrm{SDP}(\bA) - \frac{n}{2} \varepsilon,
\end{equation*}
with probability at least $1-\delta$.
\end{theorem}

\section{Numerical Experiments}\label{sec:experiments}
In this section, we evaluate the empirical performance of the BCM algorithm with respect to the Riemannian gradient ascent (RGA) and Riemannian trust region (RTR) algorithms. All algorithms are implemented on Matlab and the experiments are run on a personal computer with 2.9 GHz processor and 16 GB memory. RGA and RTR algorithms are implemented using the Manopt package \cite{manopt} with the default options and the algorithms are terminated when the maximum allowed time is achieved (can be read from the x-axis of our plots). We implement the BCM algorithm (see Algorithm \ref{alg:bcm}) with the cyclic order $(1,2,\dots,n)$. The initial iterate $\bs^0\in\reals^{n \times r}$ is the same for all algorithms and each row of $\bs^0$ is generated uniformly at random on $\S^{r-1}$. In all experiments, the cost matrix is generated as $\bA=(\bG+\bG^\T)/n$, where $G_{ij} \sim \normal (0,1)$ for all $i \neq j$, and $G_{ii}=0$ for all $i \in [n]$. We evaluate the performance of the algorithms for various values of $n$ and $r$. Empirical results illustrate the fast convergence of BCM compared to RGA and RTR. We also observe that even when the problem size is large (e.g., $n=20,000$), BCM returns a desirable solution within $\sim$10 seconds.

\begin{figure}[!tbp]
  \centering
  \begin{minipage}[b]{0.325\textwidth}
    \includegraphics[width=\textwidth]{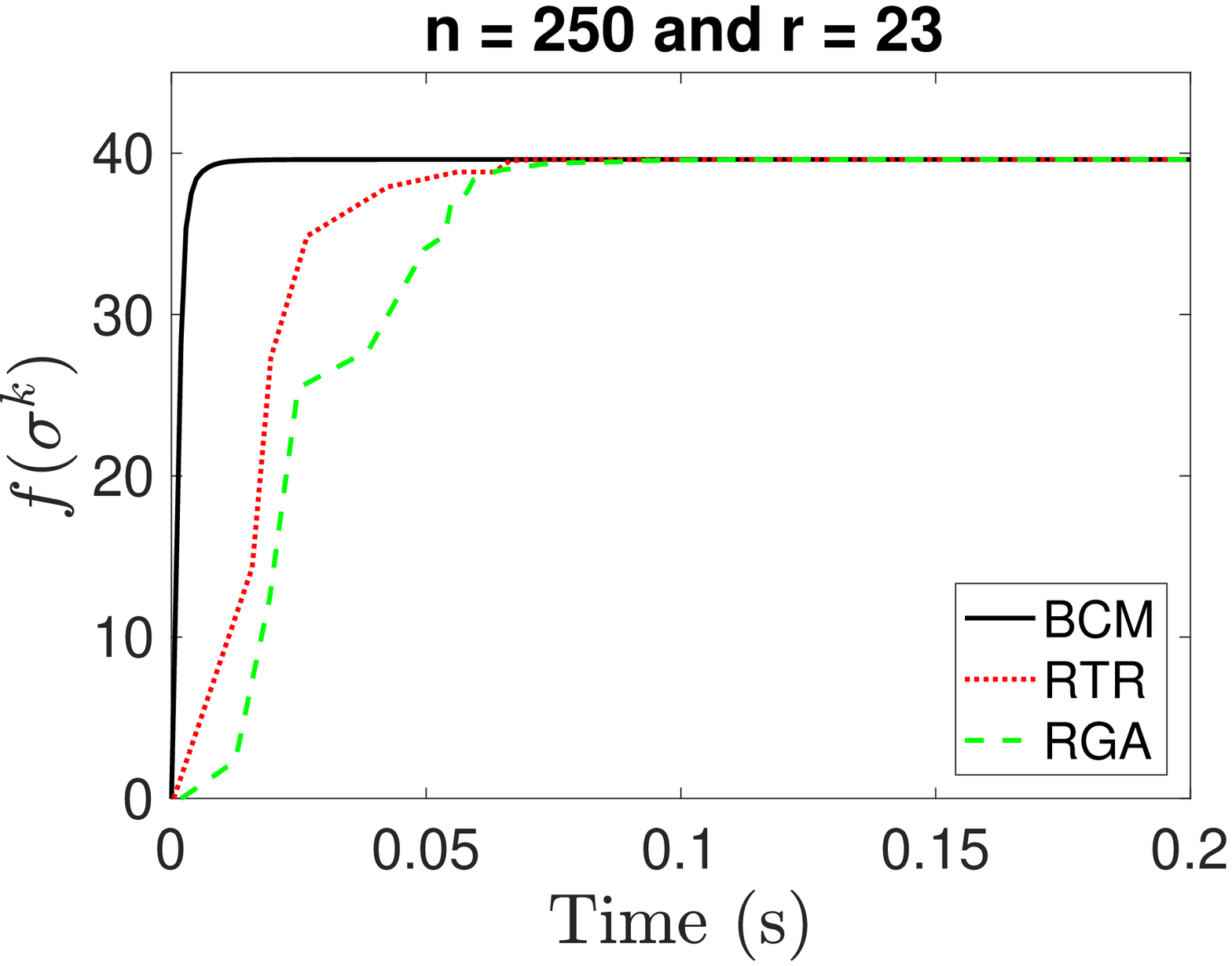}
  \end{minipage}
  \hfill
  \begin{minipage}[b]{0.325\textwidth}
    \includegraphics[width=\textwidth]{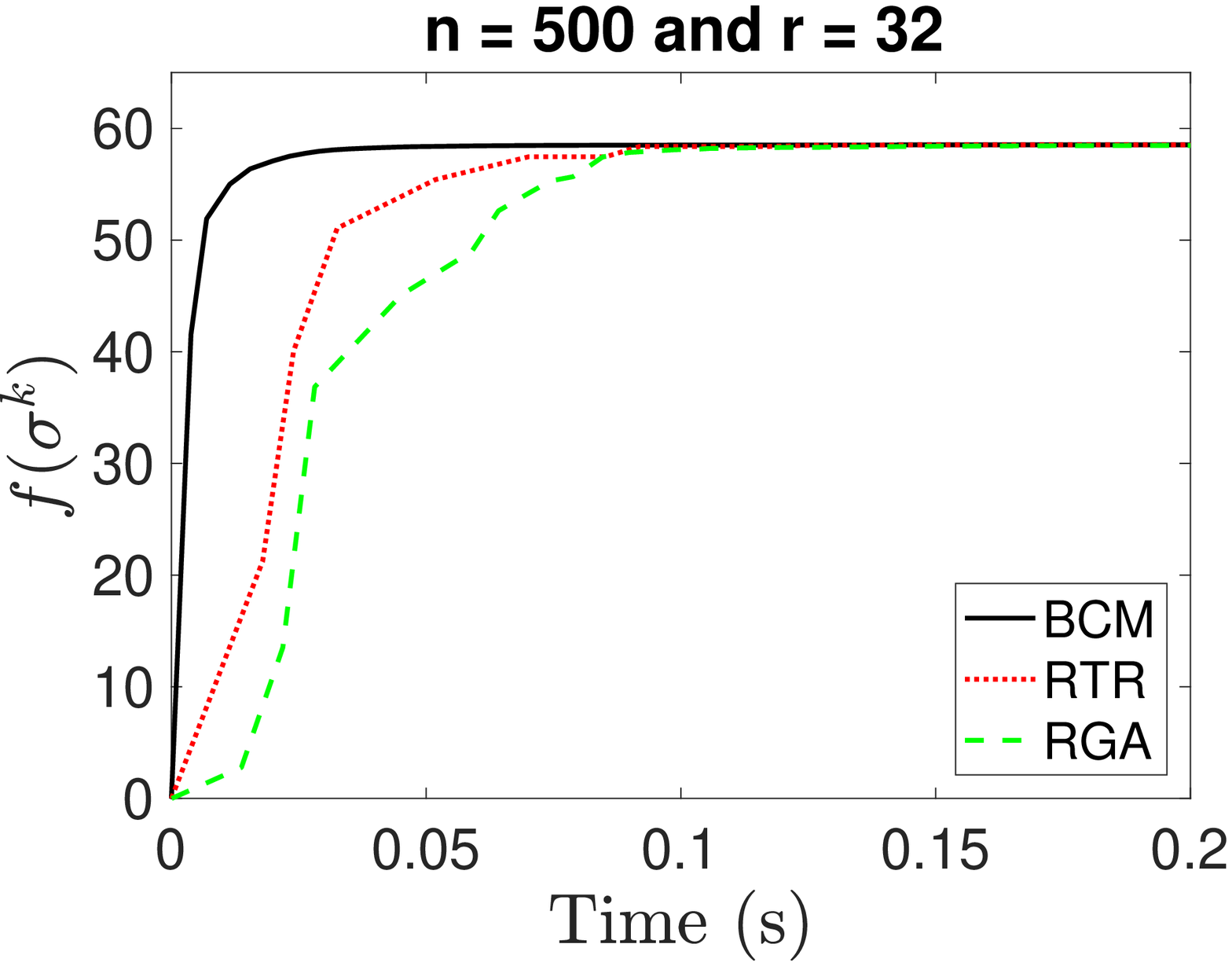}
  \end{minipage}
  \hfill
  \begin{minipage}[b]{0.325\textwidth}
    \includegraphics[width=\textwidth]{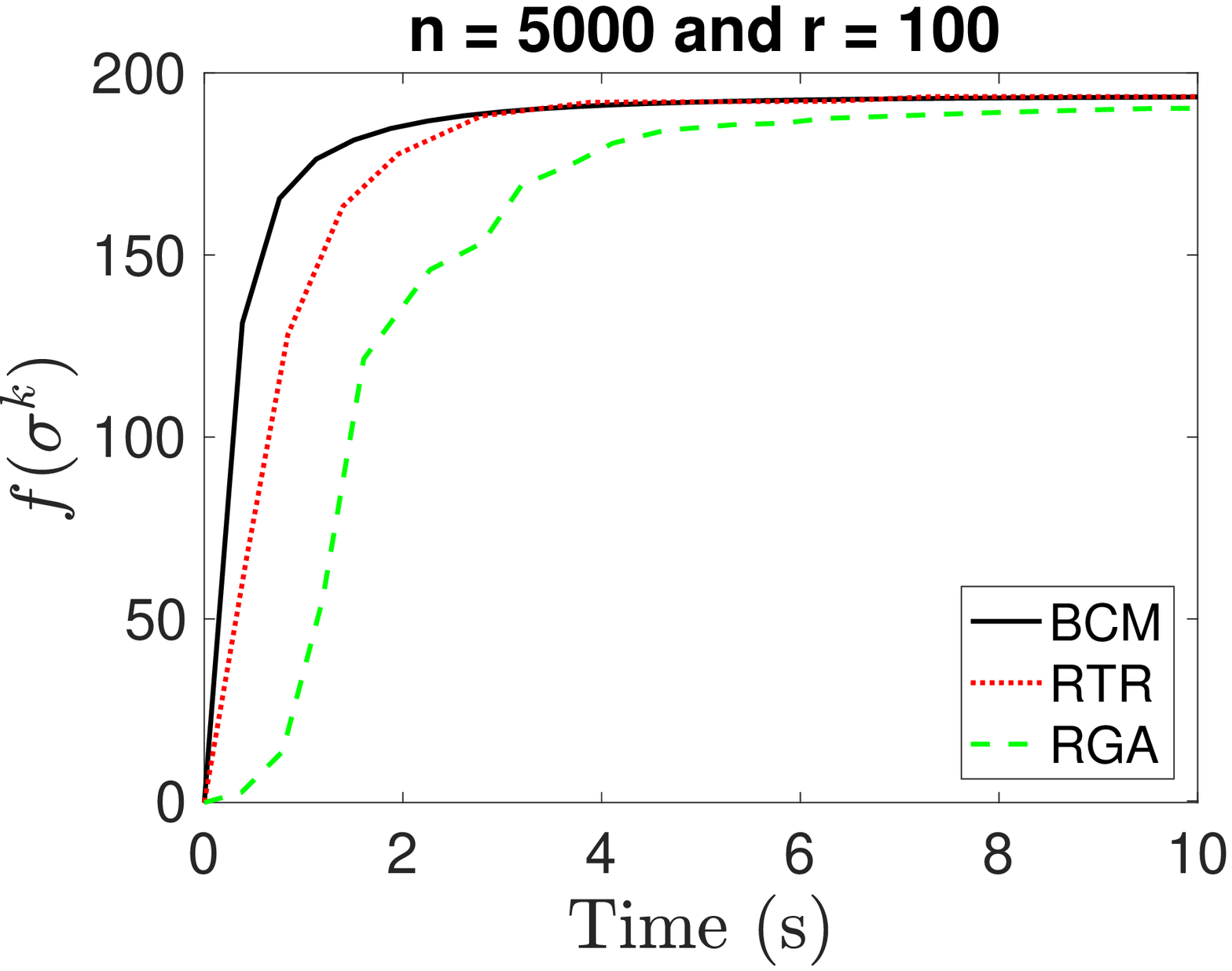}
  \end{minipage}
  \hfill
  \begin{minipage}[b]{0.325\textwidth}
    \includegraphics[width=\textwidth]{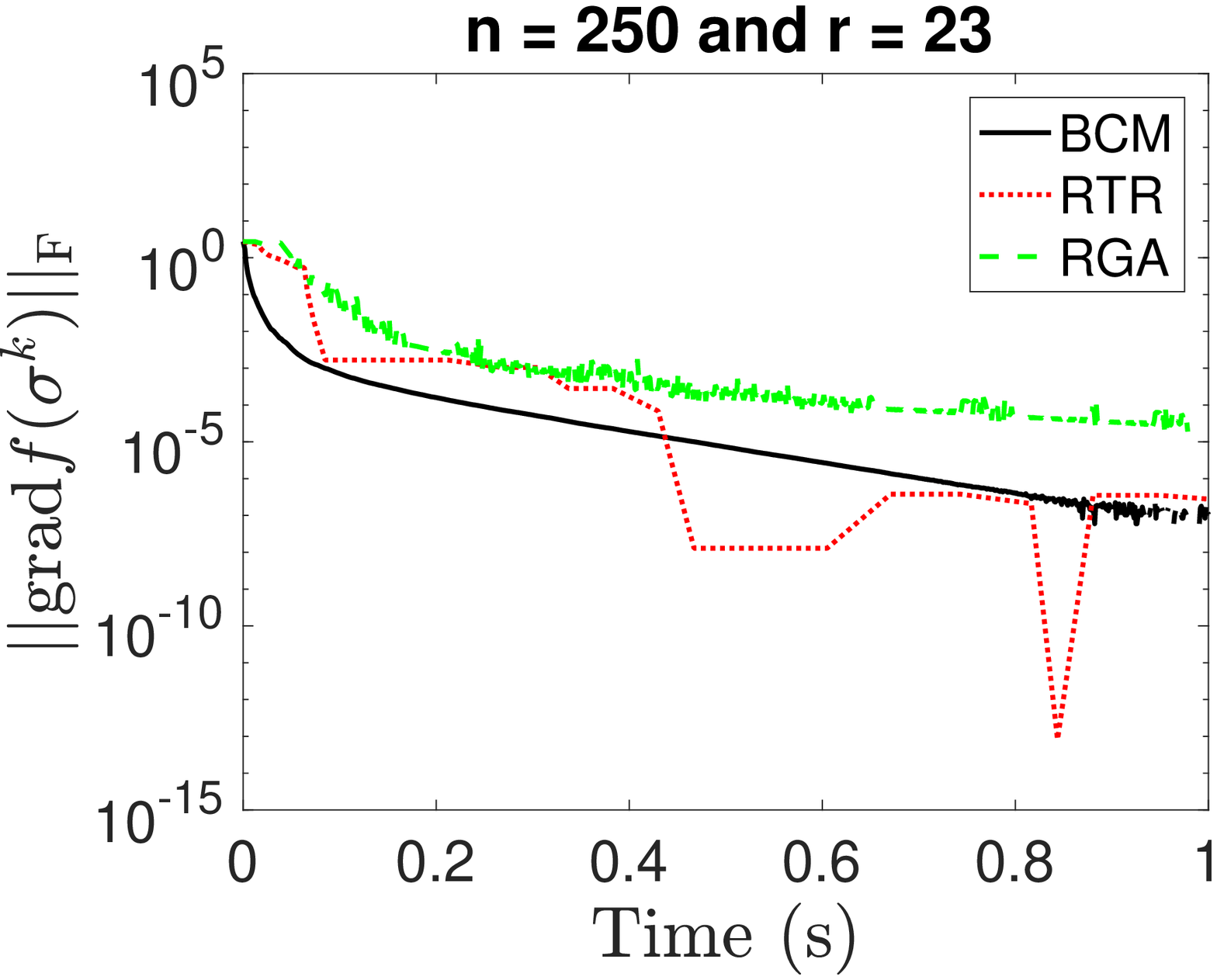}
  \end{minipage}
  \hfill
  \begin{minipage}[b]{0.325\textwidth}
    \includegraphics[width=\textwidth]{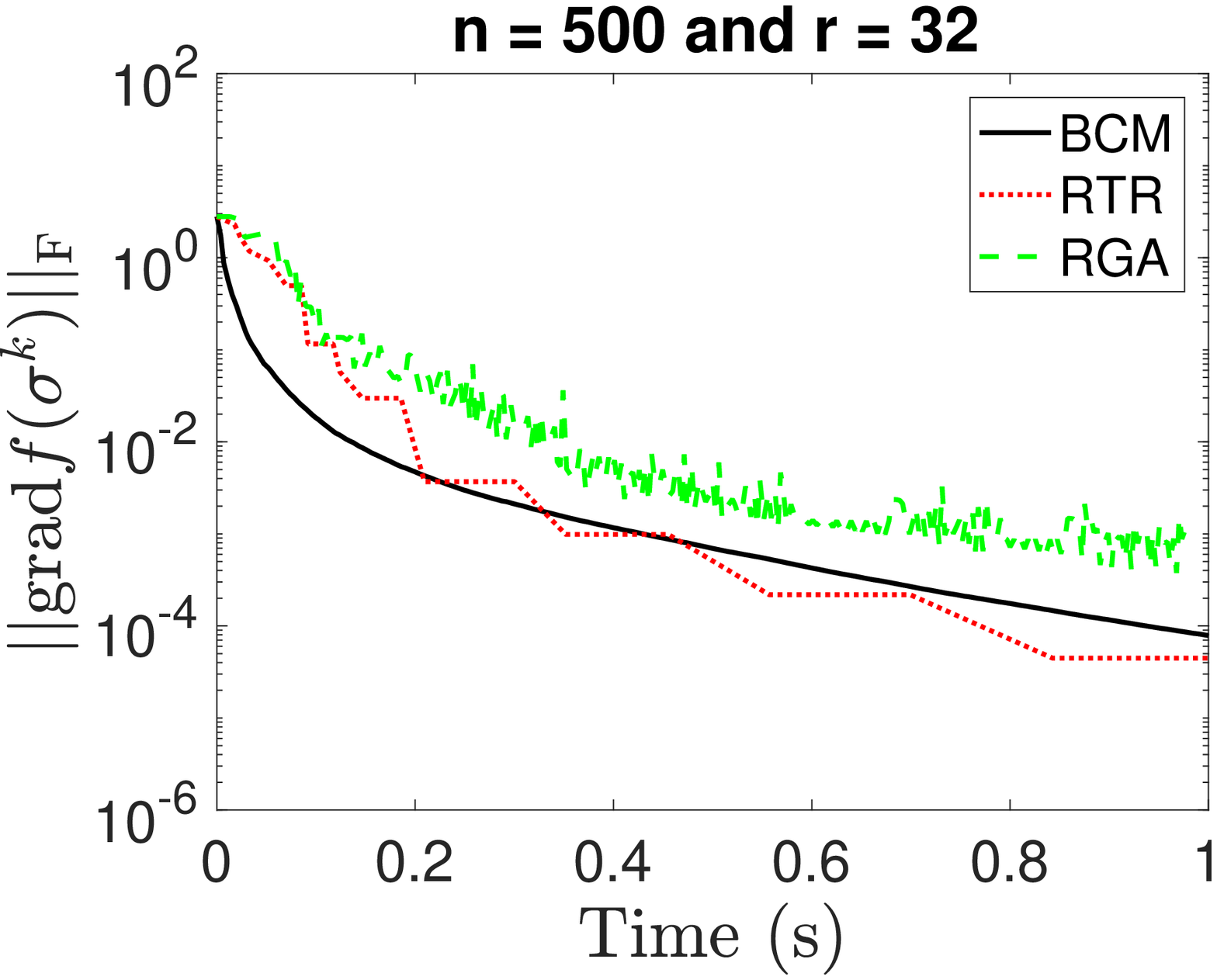}
  \end{minipage}
  \hfill
  \begin{minipage}[b]{0.325\textwidth}
    \includegraphics[width=\textwidth]{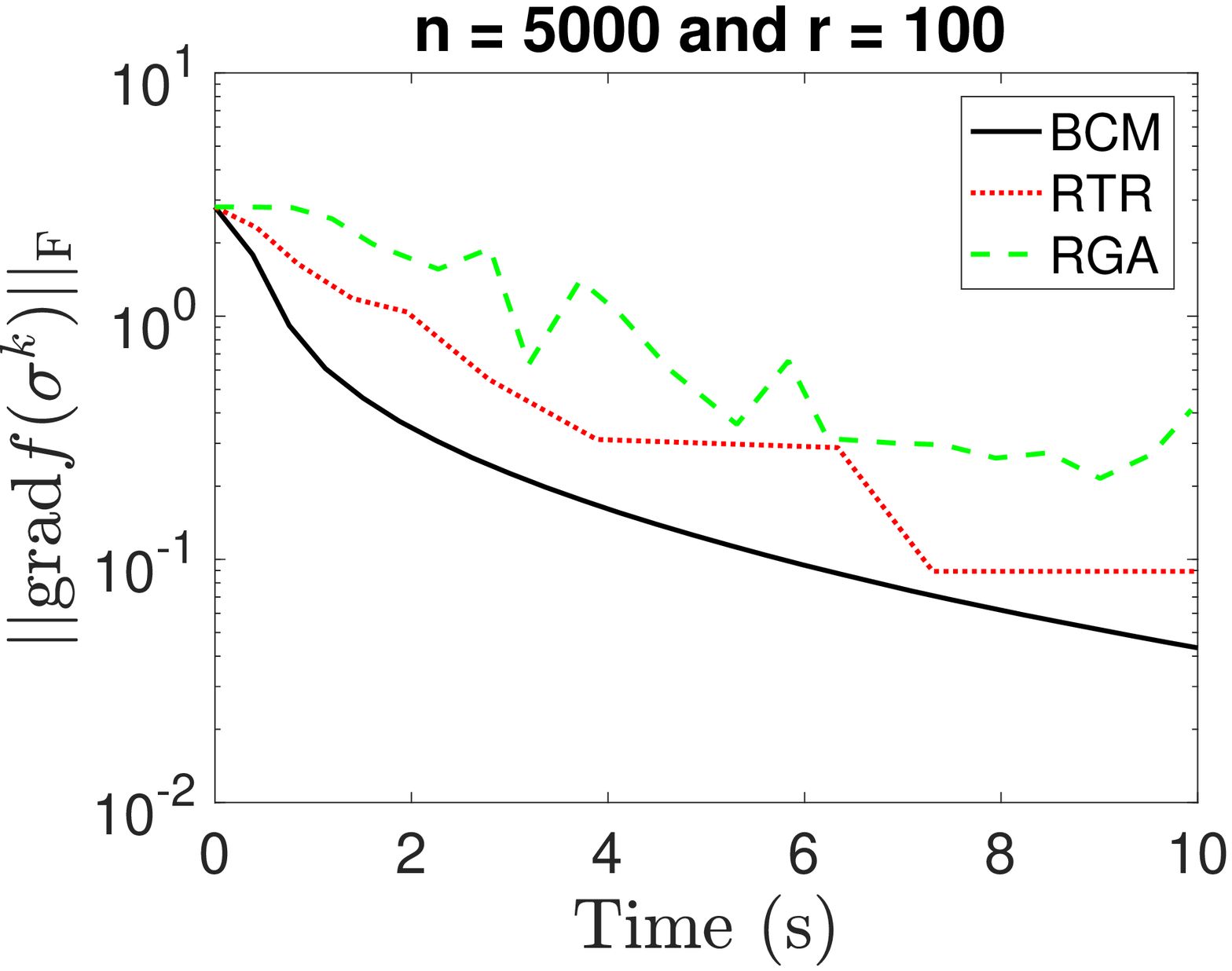}
  \end{minipage}
  \caption{Objective value and gradient norm of the BCM, RGA and RTR algorithms, when $r = \sqrt{2n}$.}
  \label{fig:small}
\end{figure}

We first consider small and moderate sized problems, i.e., $n \in \{250, 500, 5000\}$, and set $r = \lceil \sqrt{2n} \rceil$ in order to recover the solution of the SDP. For the small problems, i.e., when $n \in \{250, 500\}$, we also solve the SDP using SDPT3 implemented on Matlab. For $n=250$, the optimal value of SDP is found as $39.6147$ in $4.46$ seconds, whereas for $n=500$, the optimal value of SDP is found as $58.5327$ in $27.06$ seconds. On the other hand, Burer-Monteiro factorization based algorithms are able to return the optimal solution in less than $0.1$ seconds in both cases. Furthermore, we observe that for $n=250$, BCM returns the optimal solution in less than $0.01$ seconds, whereas RGA and RTR takes about $0.06$ seconds. Similar observation can be made for $n=500$ case as well, and we observe that BCM is about an order of magnitude faster compared to RGA and RTR. We then observe that even the moderate-sized SDPs (e.g., $n=5,000$) can be solved to optimality within $\sim$10 seconds via Burer-Monteiro approach, while \textsc{cvx} cannot solve the problem in $10$ minutes. Among the Burer-Monteiro based methods, we observe that BCM is the fastest to converge to the optimal solution. Finally, Figure \ref{fig:large} shows that even when we have a large SDP of size $n=20,000$, BCM returns a desirable solution in a few seconds, while RTR requires about $20$ seconds and RGA requires about $50$ seconds to return a solution with the same objective value.

\begin{figure}[!tbp]
  \centering
  \begin{minipage}[b]{0.49\textwidth}
    \includegraphics[width=\textwidth]{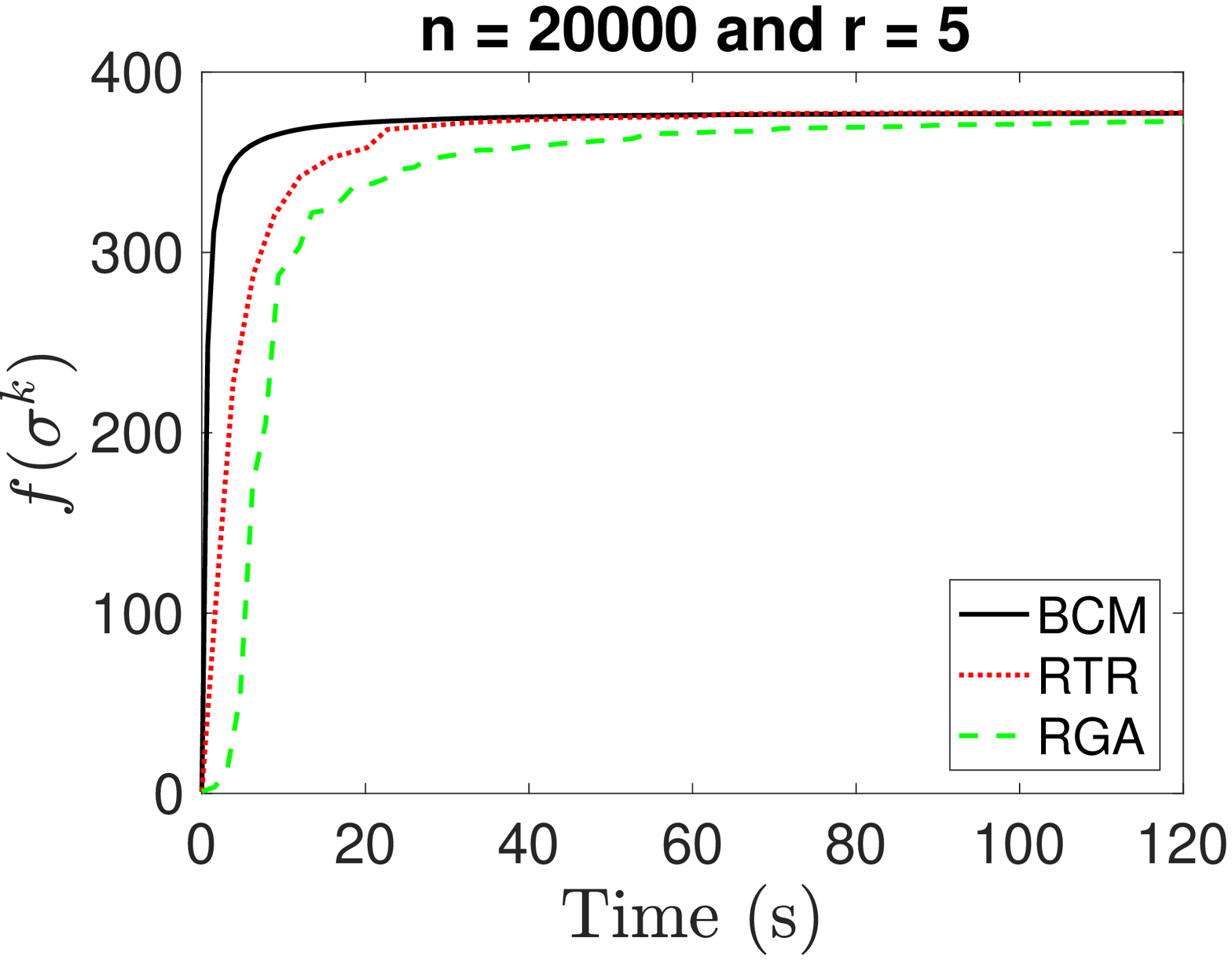}
  \end{minipage}
  \hfill
  \begin{minipage}[b]{0.49\textwidth}
    \includegraphics[width=\textwidth]{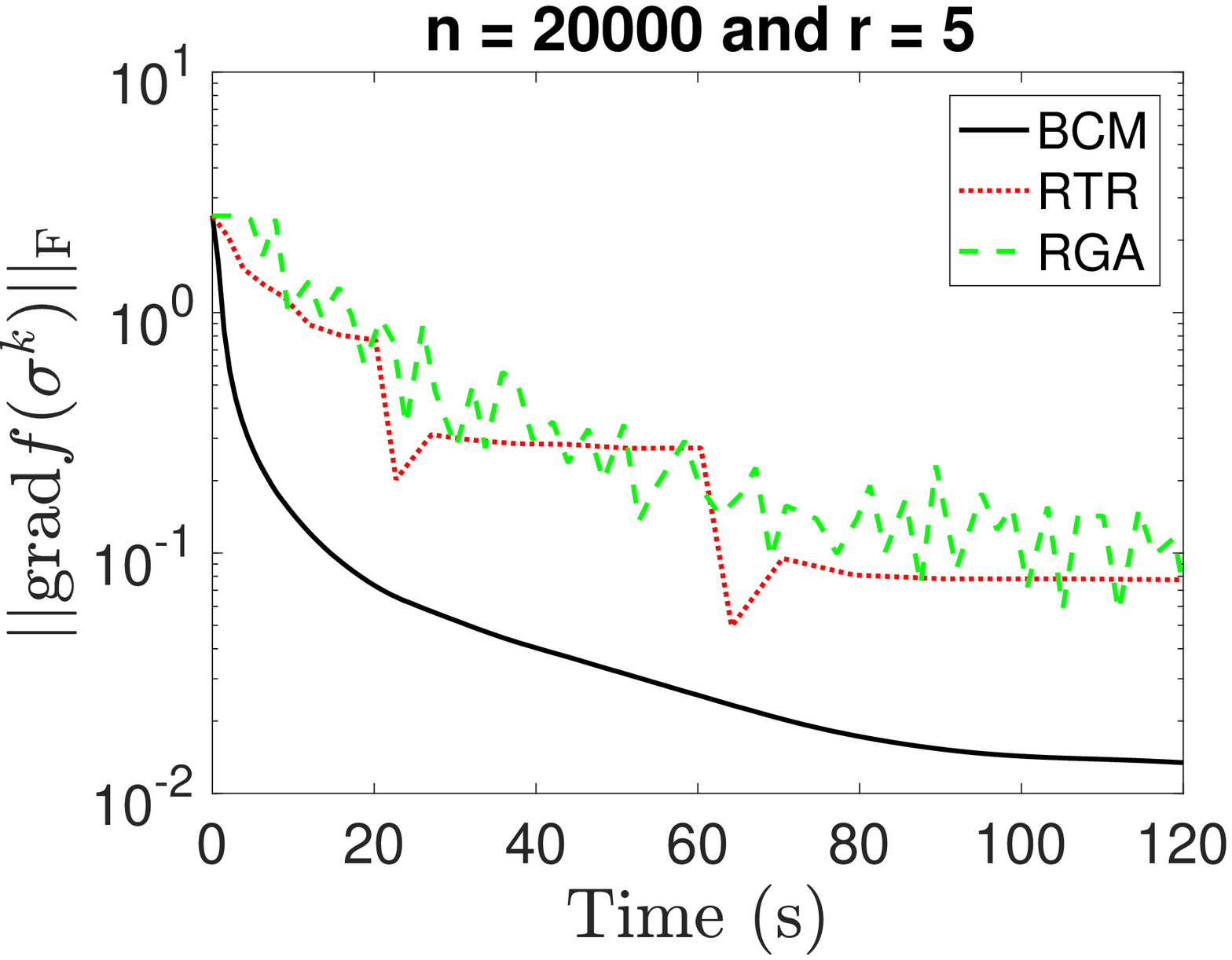}
  \end{minipage}
  \caption{Objective value and gradient norm of the BCM, RGA and RTR algorithms, when $r \ll \sqrt{2n}$.}
  \label{fig:large}
\end{figure}

\section{Conclusion}\label{sec:conclusion}
In this paper, we studied the Burer-Monteiro approach to solve large-scale SDPs. We considered to solve this non-convex problem using the block-coordinate maximization algorithm that is extremely simple to implement. We proved that for various coordinate selection rules, BCM attains a global sublinear convergence rate of $\O{1/\epsilon}$ to guarantee $\E \normf{\rgrad f(\bs^k)}^2 \leq \epsilon$. We also showed the linear convergence of BCM around a local maximum that satisfy the quadratic decay condition. We proved that the quadratic decay condition generically holds for all local maxima provided that $r \geq \sqrt{2n}$. These are the first precise rate estimates for the non-convex Burer-Monteiro approach in the literature to the best of our knowledge. We then introduced a new algorithm called BCM2 based on BCM and Lanczos methods. We showed that BCM2 is guaranteed to return a solution that provides $1-\O{1/r}$ approximation to the SDP without any assumptions on the cost matrix $\bA$, where the $r$-dependence of this approximation is optimal under the unique games conjecture. We also presented numerical results that verify our theoretical findings and show that BCM is faster than the state-of-the-art methods. Even though in this paper, we only considered SDPs with diagonal constraints, it would be of interest to study the block-coordinate maximization approach in more generic problems.

\bibliographystyle{alpha}
{\small
  \bibliography{bibliography}

\begin{thebibliography}{MMMO17}

\bibitem[ABG07]{absil2007trust}
P.-A. Absil, C.G. Baker, and K.A. Gallivan.
\newblock Trust-region methods on riemannian manifolds.
\newblock {\em Foundations of Computational Mathematics}, 7(3):303--330, Jul
  2007.

\bibitem[AHK05]{arora05}
S.~Arora, E.~Hazan, and S.~Kale.
\newblock Fast algorithms for approximate semidefiniite programming using the
  multiplicative weights update method.
\newblock In {\em Proceedings of the 46th Annual IEEE Symposium on Foundations
  of Computer Science}, FOCS '05, pages 339--348, 2005.

\bibitem[AHO97]{alizadeh97}
F.~Alizadeh, J.-P.~A. Haeberly, and M.~L. Overton.
\newblock Complementarity and nondegeneracy in semidefinite programming.
\newblock {\em Mathematical Programming}, 77(1):111--128, 1997.

\bibitem[AMS07]{absil2007manifold}
P.-A. Absil, R.~Mahony, and R.~Sepulchre.
\newblock {\em Optimization Algorithms on Matrix Manifolds}.
\newblock Princeton University Press, Princeton, NJ, USA, 2007.

\bibitem[Ani00]{anitescu00}
M.~Anitescu.
\newblock Degenerate nonlinear programming with a quadratic growth condition.
\newblock {\em SIAM Journal on Optimization}, 10(4):1116--1135, 2000.

\bibitem[BAC16]{boumal2016global}
N.~Boumal, P.-A. Absil, and C.~Cartis.
\newblock Global rates of convergence for nonconvex optimization on manifolds.
\newblock {\em arXiv preprint arXiv:1605.08101}, 2016.

\bibitem[Bar95]{barvinok95problems}
A.~I. Barvinok.
\newblock Problems of distance geometry and convex properties of quadratic
  maps.
\newblock {\em Discrete {\&} Computational Geometry}, 13(2):189--202, 1995.

\bibitem[BBV16]{boumal2016community}
A.~S. Bandeira, N.~Boumal, and V.~Voroninski.
\newblock {On the low-rank approach for semidefinite programs arising in
  synchronization and community detection}.
\newblock {\em ArXiv:1602.04426}, 2016.

\bibitem[BdOFV10]{briet10}
J.~Bri{\"e}t, F.~M. de~Oliveira~Filho, and F.~Vallentin.
\newblock The positive semidefinite grothendieck problem with rank constraint.
\newblock In {\em Automata, Languages and Programming}, pages 31--42, 2010.

\bibitem[BI95]{bonnans95}
J.~F. Bonnans and A.~Ioffe.
\newblock Second-order sufficiency and quadratic growth for nonisolated minima.
\newblock {\em Mathematics of Operations Research}, 20(4):801--817, 1995.

\bibitem[BM03]{burer2003nonlinear}
S.~Burer and R.~D.~C. Monteiro.
\newblock A nonlinear programming algorithm for solving semidefinite programs
  via low-rank factorization.
\newblock {\em Mathematical Programming}, 95(2):329--357, 2003.

\bibitem[BM05]{burer2005local}
S.~Burer and R.~D.~C. Monteiro.
\newblock Local minima and convergence in low-rank semidefinite programming.
\newblock {\em Mathematical Programming}, 103(3):427--444, Jul 2005.

\bibitem[BMAS14]{manopt}
N.~Boumal, B.~Mishra, P.-A. Absil, and R.~Sepulchre.
\newblock {M}anopt, a {M}atlab toolbox for optimization on manifolds.
\newblock {\em Journal of Machine Learning Research}, 15:1455--1459, 2014.

\bibitem[BVB16]{boumal2016non}
N.~Boumal, V.~Voroninski, and A.~S. Bandeira.
\newblock {The non-convex Burer-Monteiro approach works on smooth semidefinite
  programs}.
\newblock In {\em Advances in Neural Information Processing Systems}, pages
  2757--2765, 2016.

\bibitem[BVB18]{boumal2018deterministic}
N.~Boumal, V.~Voroninski, and A.~S. Bandeira.
\newblock Deterministic guarantees for {B}urer–{M}onteiro factorizations of
  smooth semidefinite programs.
\newblock {\em arXiv preprint arXiv:1804.02008}, 2018.

\bibitem[CR13]{coakley2013}
E.~S. Coakley and V.~Rokhlin.
\newblock A fast divide-and-conquer algorithm for computing the spectra of real
  symmetric tridiagonal matrices.
\newblock {\em Applied and Computational Harmonic Analysis}, 34(3):379 -- 414,
  2013.

\bibitem[EDM17]{erdogdu2017inference}
M.~A. Erdogdu, Y.~Deshpande, and A.~Montanari.
\newblock Inference in graphical models via semidefinite programming
  hierarchies.
\newblock In {\em Advances in Neural Information Processing Systems}, pages
  416--424, 2017.

\bibitem[GH11]{garber}
D.~Garber and E.~Hazan.
\newblock Approximating semidefinite programs in sublinear time.
\newblock In {\em Proceedings of the 24th International Conference on Neural
  Information Processing Systems}, NIPS'11, pages 1080--1088, 2011.

\bibitem[GL14]{gamarnik2014maxcut}
D.~Gamarnik and Q.~Li.
\newblock On the max-cut of sparse random graphs.
\newblock {\em arXiv preprint arXiv:1411.1698}, 2014.

\bibitem[GW95]{goemans1995improved}
M.~X. Goemans and D.~P. Williamson.
\newblock Improved approximation algorithms for maximum cut and satisfiability
  problems using semidefinite programming.
\newblock {\em Journal of the ACM (JACM)}, 42(6):1115--1145, 1995.

\bibitem[JBAS10]{journee2010lowrank}
M.~Journee, F.~Bach, P.-A. Absil, and R.~Sepulchre.
\newblock Low-rank optimization on the cone of positive semidefinite matrices.
\newblock {\em SIAM Journal on Optimization}, 20(5):2327--2351, 2010.

\bibitem[JMRT16]{javanmard2016phase}
A.~Javanmard, A.~Montanari, and F.~Ricci-Tersenghi.
\newblock Phase transitions in semidefinite relaxations.
\newblock {\em Proceedings of the National Academy of Sciences},
  113(16):E2218--E2223, 2016.

\bibitem[KL96]{klein96}
P.~Klein and H.-I Lu.
\newblock Efficient approximation algorithms for semidefinite programs arising
  from {MAX CUT} and {COLORING}.
\newblock In {\em Proceedings of the Twenty-eighth Annual ACM Symposium on
  Theory of Computing}, STOC '96, pages 338--347, New York, NY, USA, 1996. ACM.

\bibitem[KW92]{kuczynski1992}
J.~Kuczy\'{n}ski and H.~Wo\'{z}niakowski.
\newblock Estimating the largest eigenvalues by the power and lanczos
  algorithms with a random start.
\newblock {\em SIAM J. Matrix Anal. Appl.}, 13(4):1094--1122, 1992.

\bibitem[LSJR16]{lee16}
J.~D. Lee, M.~Simchowitz, M.~I. Jordan, and B.~Recht.
\newblock Gradient descent only converges to minimizers.
\newblock In {\em 29th Annual Conference on Learning Theory}, volume~49, pages
  1246--1257. PMLR, 23--26 Jun 2016.

\bibitem[MMMO17]{mei2017solving}
S.~Mei, T.~Misiakiewicz, A.~Montanari, and R.~I. Oliveira.
\newblock {Solving SDPs for synchronization and MaxCut problems via the
  Grothendieck inequality}.
\newblock {\em arXiv preprint arXiv:1703.08729}, 2017.

\bibitem[Mon16]{montanari2016grothendieck}
A.~Montanari.
\newblock {A Grothendieck-type inequality for local maxima}.
\newblock {\em arXiv preprint arXiv:1603.04064}, 2016.

\bibitem[Par03]{parrilo2003sdp}
P.~A. Parrilo.
\newblock Semidefinite programming relaxations for semialgebraic problems.
\newblock {\em Mathematical Programming}, 96(2):293--320, May 2003.

\bibitem[Pat98]{pataki1998rank}
G.~Pataki.
\newblock On the rank of extreme matrices in semidefinite programs and the
  multiplicity of optimal eigenvalues.
\newblock {\em Mathematics of operations research}, 23(2):339--358, 1998.

\bibitem[Ste10]{steurer}
D.~Steurer.
\newblock Fast {SDP} algorithms for constraint satisfaction problems.
\newblock In {\em Proceedings of the Twenty-First Annual ACM-SIAM Symposium on
  Discrete Algorithms}, pages 684--697, 2010.

\bibitem[TYUC17]{tropp17}
J.~A. Tropp, A.~Yurtsever, M.~Udell, and V.~Cevher.
\newblock Practical sketching algorithms for low-rank matrix approximation.
\newblock {\em SIAM Journal on Matrix Analysis and Applications},
  38(4):1454--1485, 2017.

\bibitem[VB96]{boyd1996sdp}
L.~Vandenberghe and S.~Boyd.
\newblock Semidefinite programming.
\newblock {\em SIAM Review}, 38(1):49--95, 1996.

\bibitem[WCK17]{wang2017mixing}
P.-W. Wang, W.-C. Chang, and J.~Z. Kolter.
\newblock The mixing method: coordinate descent for low-rank semidefinite
  programming.
\newblock {\em arXiv preprint arXiv:1706.00476}, 2017.

\end{thebibliography}
}


\newpage
\appendix

\section{Riemannian Geometry of the Problem}\label{app:riemann}
Recall the definitions of the manifold and tangent space from Section \ref{ssec:riemann}:
\begin{align*}
	\M_r & = \left\{ \bs=(\s_1, \dots, \s_n)^\T \in \reals^{n \times r} : \norm{\s_i}=1, \forall i \in [n] \right\}, \\
	T_{\bs}\M_r & = \left\{ \bu=(u_1, \dots, u_n)^\T \in \reals^{n \times r} : \inner{u_i,\s_i}=0, \forall i \in [n] \right\}.
\end{align*}
Before computing the Riemannian gradient and Hessian, we first let $\P^\perp : \reals^{n \times r} \to T_{\bs}\M_r$ denote the projection operator from the Euclidean space to the tangent space of $\bs$. When applied to a given matrix ${\boldsymbol w} = (w_1, \dots, w_n)^\T \in\reals^{n \times r}$, this projection operator yields
\begin{align*}
	\P^\perp({\boldsymbol w}) & = (\P^\perp_1(w_1), \dots, \P^\perp_n(w_n))^\T, \\
		& = (w_1-\inner{\s_1,w_1}\s_1, \dots, w_n-\inner{\s_n,w_n}\s_n)^\T, \\
		& = {\boldsymbol w} - \Diag(\diag({\boldsymbol w}\bs^\T)) \, \bs.
\end{align*}

Using this notation and standard tools from matrix manifolds \cite{absil2007manifold}, we obtain the Riemannian gradient of $f$ as follows
\begin{equation*}
	\rgrad f(\bs) = \P^\perp(\grad f(\bs)) = 2 \left(\bA - \bL \right) \bs,
\end{equation*}
where $\bL = \Diag(\diag(\bA\bs\bs^\T))$. Opening up the terms in the above equality, we obtain
\begin{align*}
	\rgrad f(\bs) & = 2 \left(\bA - \Diag \left( \diag \left( \bA 
	\begin{bmatrix}
	1 & \langle \sigma_1, \sigma_2 \rangle & \hdots & \langle \sigma_1, \sigma_n \rangle \\
	\langle \sigma_2, \sigma_1 \rangle & 1 & \hdots & \langle \sigma_2, \sigma_n \rangle \\
	\vdots & \vdots & \ddots & \vdots \\
	\langle \sigma_n, \sigma_1 \rangle & \langle \sigma_n, \sigma_2 \rangle & \hdots & 1
	\end{bmatrix}
	\right) \right) \right) \bs, \\
		& = 2 \left(\bA - 
	\begin{bmatrix}
	\langle \sigma_1, g_1 \rangle & & \\
	& \ddots & \\
	& & \langle \sigma_n, g_n \rangle
	\end{bmatrix}
	\right) \bs, \\
		& = 2
	\begin{bmatrix}
	- \langle \sigma_1, g_1 \rangle & A_{12} & \hdots & A_{1n} \\
	A_{21} & - \langle \sigma_2, g_2 \rangle & \hdots & A_{2n} \\
	\vdots & \vdots & \ddots & \vdots \\
	A_{n1} & A_{n2} & \hdots & - \langle \sigma_n, g_n \rangle
	\end{bmatrix}
	\bs.
\end{align*}
Hence, the Riemannian gradient can be explicitly written as follows
\begin{equation}
	\rgrad f(\bs) = 2 \, (g_1-\inner{\s_1,g_1}\s_1, \dots, g_n-\inner{\s_n,g_n}\s_n)^\T. \nn
\end{equation}
In particular, we have
\begin{equation}\label{eq:rgradNorm}
	\normf{\rgrad f(\bs)}^2 = 2 \sum_{i=1}^n \norm{g_i - \langle \sigma_i, g_i \rangle \sigma_i}^2 = 2 \sum_{i=1}^n \left( \norm{g_i}^2 - \langle \sigma_i, g_i \rangle^2 \right).
\end{equation}

Using the same approach, we can calculate the Riemannian Hessian of $f(\s)$ along the direction of a vector $\bu \in T_{\bs}\M_r$ by projecting the directional derivative of the gradient vector field onto the tangent space of $\bs$ as follows
\begin{equation*}
	\rhess f(\bs)[\bu] = \P^\perp\left( \dgrad f(\bs)[\bu] \right),
\end{equation*}
where $\dgrad f(\bs)[\bu]$ denotes the directional gradient of $\rgrad f(\bs)$ along the direction $\bu$. This yields
\begin{equation}\label{eq:rhessian}
	\rhess f(\bs)[\bu] = \P^\perp\left( 2(\bA-\bL)\bu - 2\ddiag{\bA \bs \bu^\T+\bA \bu \bs^\T}\bs \right) = \P^\perp\left( 2(\bA-\bL)\bu \right).
\end{equation}
In particular, we have
\begin{equation}\label{eq:rhessNorm}
	\inner{\bu,\rhess f(\bs)[\bu]} = 2 \inner{\bu,(\bA-\bL)\bu},
\end{equation}
for any $\bu \in T_{\bs}\M_r$.

The geodesics $t \mapsto \bs(t)$ (i.e., curves of shortest path with zero acceleration) can be expressed as a function of $\bs=\bs(0)\in\M_r$ and $\bu\in T_{\bs}\M_r$ as follows
\begin{equation}\label{eq:geodesics}
	\s_i(t) = \s_i \cos(\norm{u_i}t) + \frac{u_i}{\norm{u_i}} \sin(\norm{u_i}t).
\end{equation}
This geodesic can be thought as the curve on the manifold that are obtained by moving from $\bs\in\M_r$ towards the direction pointed by $\bu\in T_{\bs}\M_r$. According to this definition, the exponential map $\Exp_{\bs}: T_{\bs} \M_r \to \M_r$ corresponds to evaluating the point at $t=1$ on the geodesic function, i.e., letting $\bs' = \Exp_{\bs}(\bu)$, where $\bu\in T_{\bs} \M_r$, we have
\begin{equation*}
	\s_i' = \s_i \cos(\norm{u_i}) + \frac{u_i}{\norm{u_i}} \sin(\norm{u_i}).
\end{equation*}

These geodesics also yield the following distance function defined on the manifold
\begin{equation}\label{eq:distance}
	\dist (\bs,\bs') = \left( \sum_{i=1}^n (\arccos\inner{\s_i,\s_i'})^2 \right)^{1/2},
\end{equation}
where letting $\bs' = \Exp_{\bs}(\bu)$, we obtain
\begin{align*}
	\dist (\bs,\bs') & = \left( \sum_{i=1}^n (\arccos\inner{\s_i,\s_i \cos\norm{u_i}})^2 \right)^{1/2}, \\
		& = \left( \sum_{i=1}^n \norm{u_i}^2 \right)^{1/2}, \\
		& = \normf{\bu}.
\end{align*}
Similarly, the distance between a point $\bs$ and an equivalence class $[\bs']$ can be found as
\begin{align*}
	\dist (\bs,[\bs']) & = \min_{\bQ \in O(r)} \dist (\bs, \bs'\bQ).
\end{align*}


\section{An Ascent Lemma for BCM}
\begin{lemma}\label{lem:ascent}
Suppose at the $k$-th iteration of the BCM algorithm, $i_k$-th block is chosen (with some coordinate selection rule). Then, the BCM algorithm yields the following ascent on the objective value:
\begin{equation}
	f(\bs^{k+1}) - f(\bs^k) = 2 \left( \norm{g_{i_k}^k} - \langle \sigma_{i_k}^k, g_{i_k}^k \rangle \right) \geq 0. \nn
\end{equation}
\end{lemma}

\begin{proof}
According to the update rule of the BCM algorithm, we have $\s_{i_k}^{k+1} = \frac{g_{i_k}^k}{\norm{g_{i_k}^k}}$, which yields
\begin{align}
	f(\bs^{k+1}) & = \sum_{i=1}^n \langle \sigma_i^{k+1}, g_i^{k+1} \rangle ,\nn\\
		& = \langle \sigma_{i_k}^{k+1}, g_{i_k}^{k+1} \rangle + \sum_{i \neq i_k} \langle \sigma_i^{k+1}, g_i^{k+1} \rangle, \nn\\
		& = \norm{g_{i_k}^k} + \sum_{i \neq i_k} \langle \sigma_i^k, g_i^k - A_{ii_k} \sigma_{i_k}^k + A_{ii_k} \frac{g_{i_k}^k}{\norm{g_{i_k}^k}} \rangle, \label{eq:pol1}
\end{align}
where the last equality follows since $g_{i_k}^{k+1} = g_{i_k}^k$ as this quantity is independent of $\s_{i_k}^{k+1}$ and we have the following for all $i \neq i_k$:
\begin{align*}
	g_i^{k+1} & = \sum_{j \neq i} A_{ij} \s_i^{k+1}, \\
		& = A_{i{i_k}} \s_{i_k}^{k+1} + \sum_{j \neq i, i_k} A_{ij} \s_i^{k+1}, \\
		& = A_{i{i_k}} \frac{g_{i_k}^k}{\norm{g_{i_k}^k}} + \sum_{j \neq i, i_k} A_{ij} \s_i^k ,\\
		& = A_{i{i_k}} \frac{g_{i_k}^k}{\norm{g_{i_k}^k}}+ g_i^k - A_{i{i_k}} \s_{i_k}^k.
\end{align*}
Separating the terms in the sum in \eqref{eq:pol1} and using the fact that $A$ is a symmetric matrix, we obtain
\begin{align*}
	f(\bs^{k+1}) & = \norm{g_{i_k}^k} + \sum_{i \neq i_k} A_{ii_k} \langle \sigma_i^k, \frac{g_{i_k}^k}{\norm{g_{i_k}^k}} - \sigma_{i_k}^k \rangle + \sum_{i \neq i_k} \langle \sigma_i^k, g_i^k \rangle ,\\
		& = \norm{g_{i_k}^k} + \langle g_{i_k}^k, \frac{g_{i_k}^k}{\norm{g_{i_k}^k}} - \sigma_{i_k}^k \rangle + \sum_{i \neq i_k} \langle \sigma_i^k, g_i^k \rangle, \\
		& = 2 \norm{g_{i_k}^k} - \langle \sigma_{i_k}^k, g_{i_k}^k \rangle + \sum_{i \neq i_k} \langle \sigma_i^k, g_i^k \rangle ,\\
		& = f(\bs^k) + 2 \left( \norm{g_{i_k}^k} - \langle \sigma_{i_k}^k, g_{i_k}^k \rangle \right),
\end{align*}
which concludes the proof of the lemma.
\end{proof}


\section{Proof of Theorem \ref{thm:sublinear}}
From Lemma \ref{lem:ascent}, we have
\begin{equation*}
	f(\bs^{k+1}) - f(\bs^k) = 2 \, \left( \norm{g_{i_k}^k} - \langle \sigma_{i_k}^k, g_{i_k}^k \rangle \right)  = 2 \, \max_{i \in [n]} \left( \norm{g_i^k} - \langle \sigma_i^k, g_i^k \rangle \right),
\end{equation*}
where the last equality follows by the greedy coordinate selection rule. We can rewrite this equality as follows:
\begin{align*}
	f(\bs^{k+1}) - f(\bs^k) & = \max_{i \in [n]} \frac{2\norm{g_i^k} \left( \norm{g_i^k} - \langle \sigma_i^k, g_i^k \rangle \right)}{\norm{g_i^k}} ,\\
		& \geq \max_{i \in [n]} \frac{\norm{g_i^k}^2 - \langle \sigma_i^k, g_i^k \rangle^2}{\norm{g_i^k}},
\end{align*}
where the inequality follows since $\norm{g_i^k} \geq \inner{\s_i^k,g_i^k}$, for all $\s_i^k \in \reals^{n \times r}$. Lower bounding the maximum with the mean of the terms, we get
\begin{equation}
	f(\bs^{k+1}) - f(\bs^k) \geq \sum_{i=1}^n \frac{\norm{g_i^k}^2 - \langle \sigma_i^k, g_i^k \rangle^2}{n\norm{g_i^k}}. \label{eq:subL1}
\end{equation}
The $\norm{g_i^k}$ term in the denominator in \eqref{eq:subL1} can be upper bounded as follows
\begin{equation}
	\norm{g_{i_k}^k} \leq \sum_{j \neq i_k} |A_{i_kj}| \norm{\sigma_j^k} \leq \norm{\bA}_1. \label{eq:gBound}
\end{equation}
Using this bound in \eqref{eq:subL1}, we get
\begin{equation}
	f(\bs^{k+1}) - f(\bs^k) \geq \frac{1}{n \norm{A}_1} \, \sum_{i=1}^n  \left( \norm{g_i^k}^2 - \langle \sigma_i^k, g_i^k \rangle^2 \right) = \frac{\normf{\mathrm{grad} f(\bs^k)}^2}{2n \norm{\bA}_1}. \label{eq:subL2}
\end{equation}
In order to prove \eqref{eq:sublinearRate3}, we assume the contrary that $\normf{\mathrm{grad} f(\bs^k)}^2 > \epsilon$, for all $k\in[K-1]$. Then, using the boundedness of $f$, we observe that
\begin{equation}
	f^* - f(\bs^0) \geq f(\bs^K) - f(\bs^0) = \sum_{k=0}^{K-1} \left[ f(\bs^{k+1}) - f(\bs^k) \right]. \nn
\end{equation}
Using the functional ascent of BCM in \eqref{eq:subL2} above, we get
\begin{equation}
	f^* - f(\bs^0) \geq \sum_{k=0}^{K-1} \frac{\normf{\mathrm{grad} f(\bs^k)}^2}{2n\norm{\bA}_1} > \frac{K \epsilon}{2n\norm{\bA}_1}, \nn
\end{equation}
where the last inequality follows by the assumption. Then, by contradiction, the algorithm returns a solution with $\normf{\mathrm{grad} f(\bs^k)}^2 \leq \epsilon$, for some $k\in[K-1]$, provided that
\begin{equation}
	K \geq \frac{2n\norm{\bA}_1 (f^* - f(\bs^0))}{\epsilon}. \nn
\end{equation}

\section{Proof of Corollary \ref{cor:sublinear}}
Similar to the proof of Theorem \ref{thm:sublinear}, from Proposition \ref{lem:ascent}, we have
\begin{align}
	f(\bs^{k+1}) - f(\bs^k) & = 2 \left( \norm{g_{i_k}^k} - \langle \sigma_{i_k}^k, g_{i_k}^k \rangle \right), \nn\\
		& = \frac{2 \norm{g_{i_k}^k} \left( \norm{g_{i_k}^k} - \langle \sigma_{i_k}^k, g_{i_k}^k \rangle \right)}{\norm{g_{i_k}^k}} ,\nn\\
		& \geq \frac{ \norm{g_{i_k}^k}^2 - \langle \sigma_{i_k}^k, g_{i_k}^k \rangle^2 }{\norm{g_{i_k}^k}}, \label{eq:sublin1}
\end{align}
where the inequality follows since $\norm{g_{i_k}^k} \geq \inner{\s_{i_k}^k,g_{i_k}^k}$, for all $\s_{i_k}^k \in \reals^{n \times r}$. Letting $\E_k$ denote the expectation over $i_k$ given $\s^k$, we have
\begin{equation}
	\E_k f(\bs^{k+1}) - f(\bs^k) \geq \sum_{i=1}^n p_i \frac{ \norm{g_i^k}^2 - \langle \sigma_i^k, g_i^k \rangle^2 }{\norm{g_i^k}}. \nn
\end{equation}
In particular, when $p_i=\frac{1}{n}$, for all $i \in [n]$ (i.e., for uniform sampling case), we have
\begin{equation}
	\E_k f(\bs^{k+1}) - f(\bs^k) \geq \frac{1}{n \norm{\bA}_1} \, \sum_{i=1}^n  \left( \norm{g_i^k}^2 - \langle \sigma_i^k, g_i^k \rangle^2 \right), \nn
\end{equation}
since $\norm{g_i^k} \leq \norm{\bA}_1$, for all $i \in [n]$ by \eqref{eq:gBound}. Therefore, we have
\begin{equation}
	\E_k f(\bs^{k+1}) - f(\bs^k) \geq \frac{\normf{\mathrm{grad} f(\bs^k)}^2}{2n \norm{\bA}_1}. \label{eq:subL3}
\end{equation}
On the other hand, when $p_i=\frac{\norm{g_i^k}}{\sum_{j=1}^n \norm{g_j^k}}$ (i.e., for importance sampling case), we have
\begin{equation*}
	\E_k f(\bs^{k+1}) - f(\bs^k) \geq \frac{ \sum_{i=1}^n \norm{g_i^k}^2 - \langle \sigma_i^k, g_i^k \rangle^2 }{\sum_{j=1}^n \norm{g_j^k}} = \frac{ \normf{\mathrm{grad} f(\bs^k)}^2}{2 \sum_{j=1}^n \norm{g_j^k}}.
\end{equation*}
Letting $\norm{\bA}_{1,1} = \sum_{i,j=1}^n |\bA_{ij}|$ denote the $L_{1,1}$ norm of matrix $\bA$, we observe that $\sum_{j=1}^n \norm{g_j^k} \leq \norm{\bA}_{1,1}$, which in the above inequality yields
\begin{equation}
	\E_k f(\bs^{k+1}) - f(\bs^k) \geq \frac{\normf{\mathrm{grad} f(\bs^k)}^2}{2\norm{\bA}_{1,1}}. \label{eq:subL4}
\end{equation}
In order to prove \eqref{eq:sublinearRate2}, which corresponds to uniform sampling case, we assume the contrary that $\E \normf{\mathrm{grad} f(\bs^k)}^2 > \epsilon$, for all $k\in[K-1]$. Then, using the boundedness of $f$, we get
\begin{equation}
	f^* - f(\bs^0) \geq \E f(\bs^K) - f(\bs^0) = \sum_{k=0}^{K-1} \E \left[ f(\bs^{k+1}) - f(\bs^k) \right] = \sum_{k=0}^{K-1} \E \left[ \E_k f(\bs^{k+1}) - f(\bs^k) \right]. \nn
\end{equation}
Using the expected functional ascent of BCM in \eqref{eq:subL3} above, we get
\begin{equation}
	f^* - f(\bs^0) \geq \sum_{k=0}^{K-1} \frac{\E \normf{\mathrm{grad} f(\bs^k)}^2}{2n\norm{\bA}_1} > \frac{K \epsilon}{2n\norm{\bA}_1}, \label{eq:subL5}
\end{equation}
where the last inequality follows by the assumption. Then, by contradiction, the algorithm returns a solution with $\E \normf{\mathrm{grad} f(\bs^k)}^2 \leq \epsilon$, for some $k\in[K-1]$, provided that
\begin{equation}
	K \geq \frac{2n\norm{\bA}_1 (f^* - f(\bs^0))}{\epsilon}. \nn
\end{equation}
The proof of \eqref{eq:sublinearRate}, which corresponds to importance sampling case, can be obtained by using \eqref{eq:subL4} (instead of \eqref{eq:subL3}) in \eqref{eq:subL5}, and hence is omitted.

\section{Proof of Theorem \ref{thm:linearRate}}\label{app:linear}
By \eqref{eq:subL2}, we have the following functional ascent
\begin{equation}\label{eq:expFuncAscent2}
	f(\bs^{k+1}) - f(\bs^k) \geq \frac{\normf{\mathrm{grad} f(\bs^k)}^2}{2n\norm{\bA}_1}.
\end{equation}
In order to prove linear convergence, our aim is to show that $\normf{\mathrm{grad} f(\bs^k)}^2 \geq c ( f(\bsb)-f(\bs^k) )$, for some $0<c<2n\norm{\bA}_1$, in a neighborhood around the limit points of the iterates generated by the algorithm. Let $\bsb$ be the limit point of a subsequence $\{\bs^{k_\ell}\}_{k_\ell\geq0}$ that contains $\bs^k$. Then, we consider the solution $\bs \in [\bsb]$ such that $\bs$ is the projection of $\bs^k$ onto $[\bsb]$, i.e., $d(\bs,\bs^k) \leq d(\bs',\bs^k)$, for all $\bs' \in [\bsb]$. Then, by construction there exists $\bar{\bu} \perp \V_{\bs}$ such that $\Exp_{\bs}(\bar{\bu})=\bs^k$. In order to perform the local convergence analysis in a unified manner, we let $\bu = \bar{\bu}/\normf{\bar{\bu}}$ denote the normalized tangent vector and consider the following geodesic to describe $\bs^k$:
\begin{equation}\label{eq:linear1}
	\s_i^k = \s_i \cos(\norm{u_i}t) + \frac{u_i}{\norm{u_i}} \sin(\norm{u_i}t),
\end{equation}
where it can be observed that $t=\normf{\bar{\bu}}$ recovers the original exponential map $\bs^k = \Exp_{\bs}(\bar{\bu})$. The second order Taylor approximation to \eqref{eq:linear1} yields (note that $t=\normf{\bar{\bu}}<1$, when $\bs$ and $\bs^k$ are sufficiently close):
\begin{equation}
	\s_i^k = \s_i + t u_i - \frac{t^2}{2} \norm{u_i}^2 \s_i + O(t^3), \nn
\end{equation}
and using this approximation, we obtain
\begin{equation*}
	g_i^k = g_i + t v_i - \frac{t^2}{2} \gt_i + O(t^3) ,
\end{equation*}
where
\begin{equation*}
	v_i^k = \sum_{j \neq i} A_{ij} u_j \quad \text{and} \quad \gt_i = \sum_{j \neq i} A_{ij} \norm{u_j}^2 \s_j.
\end{equation*}
This yields the following Taylor approximation to $\normf{\mathrm{grad} f(\bs^k)}^2$:
\begin{align*}
	\normf{\mathrm{grad} f(\bs^k)}^2 & = 2 \sum_{i=1}^n \left( \norm{g_i^k}^2 - \inner{\s_i^k, g_i^k}^2 \right) \\
		& \hspace{-2cm} = 2 \sum_{i=1}^n \left( \norm{g_i + t v_i - \frac{t^2}{2} \gt_i}^2 - \inner{\s_i + t u_i - \frac{t^2}{2} \norm{u_i}^2 \s_i, \, g_i + t v_i - \frac{t^2}{2} \gt_i}^2 \right) + O(t^3) ,\\
		& \hspace{-2cm} = 2 \sum_{i=1}^n \left\{ \norm{g_i}^2 + 2t\inner{g_i,v_i} - t^2\inner{g_i,\gt_i} + t^2\norm{v_i}^2 \right. \\
	 		& \hspace{-1.5cm} \left. - \left( \inner{\s_i,g_i} + t \inner{\s_i,v_i} - \frac{t^2}{2} \inner{\s_i,\gt_i} + t \inner{u_i,g_i} + t^2\inner{u_i,v_i} - \frac{t^2}{2} \norm{u_i}^2 \inner{\s_i,g_i} \right)^2 \right\} + O(t^3).
\end{align*}
Note that we have $\s_i = g_i / \norm{g_i}$ due to the property of fixed points of the BCM algorithm and also $\inner{\s_i,u_i}=0$ as $u \in T_{\bs}\M_r$. Using these relations in the above equality, we get
\begin{align}
	\normf{\mathrm{grad} f(\bs^k)}^2 & = 2 \sum_{i=1}^n \left\{ \norm{g_i}^2 + 2t\norm{g_i}\inner{\s_i,v_i} - t^2\norm{g_i}\inner{\s_i,\gt_i} + t^2\norm{v_i}^2 \right. \nn\\
	 		& \hspace{.5cm} \left. - \left( \norm{g_i} + t \inner{\s_i,v_i} - \frac{t^2}{2} \inner{\s_i,\gt_i} + t^2\inner{u_i,v_i} - \frac{t^2}{2} \norm{u_i}^2 \norm{g_i} \right)^2 \right\} + O(t^3), \nn\\
	 	& = 2 \sum_{i=1}^n \left\{ \norm{g_i}^2 + 2t\norm{g_i}\inner{\s_i,v_i} - t^2\norm{g_i}\inner{\s_i,\gt_i} + t^2\norm{v_i}^2 \right. \nn\\
	 		& \hspace{.5cm} - \left( \norm{g_i}^2 + 2t \norm{g_i} \inner{\s_i,v_i} - t^2 \norm{g_i} \inner{\s_i,\gt_i} + 2t^2 \norm{g_i} \inner{u_i,v_i} \right. \nn\\
	 		& \hspace{1.5cm} \left. \left. - t^2 \norm{u_i}^2 \norm{g_i}^2 + t^2\inner{\s_i,v_i}^2 \right) \right\} + O(t^3), \nn\\
	 	& = 2 t^2 \sum_{i=1}^n \left( \norm{v_i}^2 - \inner{\s_i,v_i}^2 - 2\norm{g_i} \inner{u_i,v_i} + \norm{u_i}^2 \norm{g_i}^2 \right) + O(t^3). \label{eq:linear2}
\end{align}
Since $\inner{\s_i,u_i}=0$, we have by the Pythagorean theorem that
\begin{equation*}
	\norm{v_i}^2 - \inner{\s_i,v_i}^2 - \inner{\frac{u_i}{\norm{u_i}},v_i}^2 \geq 0.
\end{equation*}
Using this inequality in \eqref{eq:linear2}, we get
\begin{align}
	\normf{\mathrm{grad} f(\bs^k)}^2 & \geq 2 t^2 \sum_{i=1}^n \left( \inner{\frac{u_i}{\norm{u_i}},v_i}^2 - 2\norm{g_i} \inner{u_i,v_i} + \norm{u_i}^2 \norm{g_i}^2 \right) + O(t^3), \nn\\
		& = 2 t^2 \sum_{i=1}^n \left( \norm{u_i} \norm{g_i} - \inner{\frac{u_i}{\norm{u_i}},v_i} \right)^2 + O(t^3). \label{eq:linear3}
\end{align}
In order to lower bound \eqref{eq:linear3} by $c(f(\bs)-f(\bs^k))$, we consider the second order Taylor approximation of $f(\bs^k)$, which can be written as follows
\begin{align*}
	f(\bs^k) & = \sum_{i=1}^n \inner{\s_i^k, g_i^k}, \\
		& = \sum_{i=1}^n \inner{\s_i + t u_i - \frac{t^2}{2} \norm{u_i}^2 \s_i, \, g_i + t v_i - \frac{t^2}{2} \gt_i} + O(t^3), \\
		& = \sum_{i=1}^n \left( \inner{\s_i,g_i} + t \inner{\s_i,v_i} - \frac{t^2}{2} \inner{\s_i,\gt_i} + t \inner{u_i,g_i} + t^2\inner{u_i,v_i} - \frac{t^2}{2} \norm{u_i}^2 \inner{\s_i,g_i} \right) + O(t^3).
\end{align*}
Similar to the previous derivations, using the fact that $\s_i = \frac{g_i}{\norm{g_i}}$ and $\inner{\s_i,u_i}=0$, for all $i\in[n]$, we obtain
\begin{align*}
	f(\bs^k) & = f(\bs) + \sum_{i=1}^n \left( t \inner{\s_i,v_i} - \frac{t^2}{2} \inner{\s_i,\gt_i} + t^2\inner{u_i,v_i} - \frac{t^2}{2} \norm{u_i}^2 \inner{\s_i,g_i} \right) + O(t^3), \\
		& = f(\bs) + \sum_{i=1}^n \left( t \sum_{j \neq i} A_{ij} \inner{\s_i,u_j} - \frac{t^2}{2} \sum_{j \neq i} A_{ij} \norm{u_j}^2 \inner{\s_i,\s_j} + t^2\inner{u_i,v_i} - \frac{t^2}{2} \norm{u_i}^2 \inner{\s_i,g_i} \right) + O(t^3), \\
		& = f(\bs) + t \sum_{j=1}^n \sum_{i \neq j} A_{ji} \inner{\s_i,u_j} - \frac{t^2}{2} \sum_{j=1}^n \sum_{i \neq j} A_{ji} \norm{u_j}^2 \inner{\s_i,\s_j} \\
			& \hspace{2cm} + t^2 \sum_{i=1}^n \left( \inner{u_i,v_i} - \frac{1}{2} \norm{u_i}^2 \inner{\s_i,g_i} \right) + O(t^3),
\end{align*}
where the last line follows since $\bA$ is symmetric. Using the definition $g_j = \sum_{i \neq j} A_{ji} \s_i$ and $\s_i = \frac{g_i}{\norm{g_i}}$ in the above inequality yields
\begin{align}
	f(\bs^k) & = f(\bs) + t \sum_{j=1}^n \inner{g_j,u_j} - \frac{t^2}{2} \sum_{j=1}^n \norm{u_j}^2 \inner{g_j,\s_j} + t^2 \sum_{i=1}^n \left( \inner{u_i,v_i} - \frac{1}{2} \norm{u_i}^2 \inner{\s_i,g_i} \right) + O(t^3), \nonumber\\
		& = f(\bs) + t^2 \sum_{i=1}^n \left( \inner{u_i,v_i} - \norm{u_i}^2 \norm{g_i} \right) + O(t^3). \label{eq:linearRef1}
\end{align}
Reorganizing terms, we get
\begin{equation}
	f(\bsb) - f(\bs^k) = f(\bs) - f(\bs^k) = t^2 \sum_{i=1}^n \left( \norm{u_i}^2 \norm{g_i} - \inner{u_i,v_i} \right) + O(t^3). \label{eq:linear4}
\end{equation}
Turning back our attention to \eqref{eq:linear3}, we can lower bound the right-hand side as follows
\begin{align*}
	\normf{\mathrm{grad} f(\bs^k)}^2 & \geq 2 t^2 \sum_{i=1}^n \frac{1}{\norm{u_i}^2} \left( \norm{u_i} ^2\norm{g_i} - \inner{u_i,v_i} \right)^2 + O(t^3), \\
		& \geq 2 t^2 \sum_{i=1}^n \left( \norm{u_i} ^2\norm{g_i} - \inner{u_i,v_i} \right)^2 + O(t^3), \\
		& \geq \frac{2 t^2}{n} \left( \sum_{i=1}^n \left( \norm{u_i} ^2\norm{g_i} - \inner{u_i,v_i} \right) \right)^2 + O(t^3),
\end{align*}
where the second inequality follows since $\norm{u_i}^2 \leq \normf{\bu}^2 = 1$ and the last inequality follows since $\left( \sum_{i=1}^n a_i \right)^2 \leq n \sum_{i=1}^n a_i^2$, for all $a_i \in \reals$, $i \in [n]$. Using the second order approximation derived in \eqref{eq:linear4} in the above inequality, we obtain
\begin{align*}
	\normf{\mathrm{grad} f(\bs^k)}^2 & \geq \frac{f(\bsb) - f(\bs^k)}{n} \sum_{i=1}^n 2 \left( \norm{u_i} ^2\norm{g_i} - \inner{u_i,v_i} \right) + O(t^3), \\
		& = \frac{2\inner{\bu,(\bL-\bA)\bu}}{n} \left( f(\bsb) - f(\bs^k) \right) + O(t^3),
\end{align*}
where $\bL = \Diag(\norm{g_1},\dots,\norm{g_n})$. Since we have $2\inner{\bu,(\bA-\bL)\bu} \leq -\mu \normf{\bu}^2$ by the quadratic decay condition, we conclude that
\begin{equation}
	\normf{\mathrm{grad} f(\bs^k)}^2 \geq \frac{\mu}{n} \left( f(\bsb) - f(\bs^k) \right) + O(t^3). \label{eq:linear5}
\end{equation}
This implies that whenever $\s^k$ is sufficiently close to $\s$, i.e., whenever $t$ is sufficiently small (cf. \eqref{eq:linear1}), the remainder in the Taylor approximation, i.e., the $O(t^3)$ terms, will be dominated by $\frac{\mu}{n} \left( f(\bsb) - f(\bs^k) \right)$. In particular, if $\bs^0$ is sufficiently close to $\bsb$ to satisfy $O(t^3) \geq - \frac{\mu}{2n} \left( f(\bsb) - f(\bs^k) \right)$ in the above inequality, we then have
\begin{equation}\label{eq:linearRef6}
	\normf{\mathrm{grad} f(\bs^k)}^2 \geq \frac{\mu}{2n} \left( f(\bsb) - f(\bs^k) \right).
\end{equation}
Combining this inequality with \eqref{eq:expFuncAscent2}, we get
\begin{equation}\label{eq:linearRef7}
	f(\bs^{k+1}) - f(\bs^k) \geq \frac{\mu}{4n^2\norm{\bA}_1} \left( f(\bsb) - f(\bs^k) \right).
\end{equation}
Rearranging terms in the above inequality concludes the proof.

In order to quantify how close $\bs^0$ and $\bs$ should be so that this convergence rate holds, we need to derive explicit bounds on the higher order terms in \eqref{eq:linear3} and \eqref{eq:linear4}, which we do in the following. The Taylor expansion of $\bs^k$ around $\bs$ yields
\begin{align*}
	\s_i^k & = \s_i \cos(\norm{u_i}t) + \frac{u_i}{\norm{u_i}} \sin(\norm{u_i}t), \\
		& = \s_i \left[ \sum_{\ell=0}^\infty \frac{(-1)^\ell}{(2\ell)!} \left( \norm{u_i}t \right)^{2\ell} \right] + \frac{u_i}{\norm{u_i}} \left[ \sum_{\ell=0}^\infty \frac{(-1)^\ell}{(2\ell+1)!} \left( \norm{u_i}t \right)^{2\ell+1} \right].
\end{align*}
Using this expansion, we can compute $f(\bs^k) = \sum_{i,j=1}^n A_{ij} \inner{\s_i^k, \s_j^k}$. The first three terms in the expansion are already given in \eqref{eq:linearRef1} as follows
\begin{equation}\label{eq:linearRef2}
	f(\bs^k) = f(\bs) + t^2 \sum_{i=1}^n \left( \inner{u_i,v_i} - \norm{u_i}^2 \norm{g_i} \right) + \beta_f,
\end{equation}
where $\beta_f$ represents the higher order terms. In order to find an upper bound on $|\beta_f|$, we use the Cauchy-Schwarz inequality in the higher order terms in the expansion of $f(\bs^k)$, which yields the following bound
\begin{equation*}
	|\beta_f| \leq \sum_{i,j=1}^n |A_{ij}| \left( \sum_{\ell=3}^\infty \frac{t^\ell}{\ell!} ( \norm{u_i} + \norm{u_j} )^\ell \right).
\end{equation*}
As $\normf{\bu}=1$, we have $\norm{u_i}\leq1$ for all $i\in[n]$, which implies
\begin{equation*}
	|\beta_f| \leq \sum_{i,j=1}^n |A_{ij}| \left( \sum_{\ell=3}^\infty \frac{t^\ell}{\ell!} 2^\ell \right),
\end{equation*}
where we note that $t$ denotes the geodesic distance between $\bs^k$ and $[\bsb]$ as highlighted in \eqref{eq:linear1}. Assuming that $t\leq1$, we obtain the following upper bound
\begin{equation*}
	|\beta_f| \leq t^3 n \norm{\bA}_1 \left( \sum_{\ell=3}^\infty \frac{2^\ell}{\ell!} \right).
\end{equation*}
Using the inequality $\sum_{\ell=3}^\infty \frac{2^\ell}{\ell!} = e^2-5 \leq 5/2$ above, we get
\begin{equation*}
	|\beta_f| \leq \frac{5n \norm{\bA}_1t^3 }{2}.
\end{equation*}
Plugging this value back in \eqref{eq:linearRef2}, we obtain
\begin{equation}\label{eq:linearRef4}
	f(\bs^k) \leq f(\bs) + t^2 \sum_{i=1}^n \left( \inner{u_i,v_i} - \norm{u_i}^2 \norm{g_i} \right) + \frac{5n \norm{\bA}_1t^3 }{2}.
\end{equation}

Considering the same expansion for $\normf{\rgrad f(\bs^k)}^2 = 2 \sum_{i=1}^n (\norm{g_i^k}^2 - \inner{\s_i^k,g_i^k}^2)$, we get the following (see \eqref{eq:linear3}):
\begin{equation}\label{eq:linearRef3}
	\normf{\rgrad f(\bs^k)}^2 = 2 t^2 \sum_{i=1}^n \left( \norm{u_i} \norm{g_i} - \inner{\frac{u_i}{\norm{u_i}},v_i} \right)^2 + \beta_g,
\end{equation}
where $\beta_g$ represents the higher order terms. Upper bounding each higher order terms using the Cauchy-Schwarz inequality as follows, we obtain
\begin{equation*}
	|\beta_g| \leq 2 \sum_{i=1}^n \left[ \sum_{j,m=1}^n |A_{ij}| |A_{im}| \left( \sum_{\ell=3}^\infty \frac{t^\ell}{\ell!} ( \norm{u_j} + \norm{u_m} )^\ell \right) + \sum_{j,m=1}^n |A_{ij}| |A_{im}| \left( \sum_{\substack{\ell,s=0 \\ \ell+s\geq3}}^\infty \frac{t^{\ell+s}}{\ell! s!} ( \norm{u_i} + \norm{u_j} )^{\ell+s} \right) \right].
\end{equation*}
Using the fact that $\norm{u_i}\leq1$ for all $i\in[n]$, we get the following upper bound
\begin{equation*}
	|\beta_g| \leq 2 \sum_{i=1}^n \left[ \sum_{j,m=1}^n |A_{ij}| |A_{im}| \left( \sum_{\ell=3}^\infty \frac{t^\ell}{\ell!} 2^\ell \right) + \sum_{j,m=1}^n |A_{ij}| |A_{im}| \left( \sum_{\substack{\ell,s=0 \\ \ell+s\geq3}}^\infty \frac{t^{\ell+s}}{\ell! s!} 2^{\ell+s} \right) \right].
\end{equation*}
Using the upper bound $\sum_{j,m=1}^n |A_{ij}| |A_{im}| \leq \norm{\bA}_1^2$ above, we obtain
\begin{equation*}
	|\beta_g| \leq 2 \norm{\bA}_1^2 \sum_{i=1}^n \left[ \sum_{\ell=3}^\infty \frac{t^\ell}{\ell!} 2^\ell + \sum_{\substack{\ell,s=0 \\ \ell+s\geq3}}^\infty \frac{t^{\ell+s}}{\ell! s!} 2^{\ell+s} \right].
\end{equation*}
Introducing a change of variables in the last sum, we get
\begin{align*}
	|\beta_g| & \leq 2 \norm{\bA}_1^2 \sum_{i=1}^n \left[ \sum_{\ell=3}^\infty \frac{t^\ell}{\ell!} 2^\ell + \sum_{\ell=3}^\infty \frac{t^\ell}{\ell!} 2^\ell \left( \sum_{s=0}^\ell \frac{\ell!}{s!(\ell-s)!} \right) \right] ,\\
		& = 2 \norm{\bA}_1^2 \sum_{i=1}^n \left[ \sum_{\ell=3}^\infty \frac{t^\ell}{\ell!} \left( 2^\ell + 4^\ell \right) \right].
\end{align*}
Assuming that $t\leq1$, we obtain the following upper bound
\begin{equation*}
	|\beta_g| \leq 2 \norm{\bA}_1^2 t^3 \sum_{i=1}^n \left[ \sum_{\ell=3}^\infty \frac{1}{\ell!} \left( 2^\ell + 4^\ell \right) \right].
\end{equation*}
Using the inequality $\sum_{\ell=3}^\infty \frac{2^\ell+4^\ell}{\ell!} = e^2+e^4-18 \leq 44$ above, we get
\begin{equation*}
	|\beta_g| \leq 88 n \norm{\bA}_1^2 t^3.
\end{equation*}
Plugging this value back in \eqref{eq:linearRef3}, we obtain
\begin{equation}\label{eq:linearRef5}
	\normf{\rgrad f(\bs^k)}^2 \geq 2 t^2 \sum_{i=1}^n \left( \norm{u_i} \norm{g_i} - \inner{\frac{u_i}{\norm{u_i}},v_i} \right)^2 - 88 n \norm{\bA}_1^2 t^3.
\end{equation}

Using the same bounding technique as in \eqref{eq:linear5}, we get
\begin{align*}
	\normf{\mathrm{grad} f(\bs^k)}^2 & \geq \frac{\mu}{n} \left( f(\bsb) - f(\bs^k) - \frac{5n \norm{\bA}_1t^3 }{2} \right) - 88 n \norm{\bA}_1^2 t^3, \\
		& = \frac{\mu}{n} \left( f(\bsb) - f(\bs^k) \right) - t^3 \norm{\bA}_1 \left( 3\mu + 88 n \norm{\bA}_1 \right).
\end{align*}
Therefore, in order for \eqref{eq:linearRef6} to hold, we need
\begin{equation*}
	t^3	\norm{\bA}_1 \left( 3\mu + 88 n \norm{\bA}_1 \right) \leq \frac{\mu}{2n} \left( f(\bsb) - f(\bs^k) \right),
\end{equation*}
which can be equivalently rewritten as follows
\begin{equation*}
	t^3 \leq \frac{\mu (f(\bsb) - f(\bs^k))}{2n \norm{\bA}_1 \left( 3\mu + 88 n \norm{\bA}_1 \right)}.
\end{equation*}
As $f(\bs^k)$ is a monotonically non-decreasing sequence, then as soon as $\bs^0$ is sufficiently close to $[\bsb]$ in the sense that
\begin{equation*}
	\mathrm{dist}(\bs^0, [\bsb]) \leq \left( \frac{\mu (f(\bsb) - f(\bs^k))}{2n \norm{\bA}_1 \left( 3\mu + 88 n \norm{\bA}_1 \right)} \right)^{1/3},
\end{equation*}
then the linear convergence rate presented in \eqref{eq:linearRef7} holds.

\section{Proof of Theorem \ref{lem:asmpGeneral}}
Suppose $\bX^* = \bs\bs^\T$, where $\bs \in \M_r$ and $\mathrm{rank}(\bX^*) = r^* \leq r$. We have $\bZ^* = \bL-\bA$ by the definition of the dual problem, where $\bL = \Diag(\norm{g_1},\dots,\norm{g_n})$. Due to strict complementarity, we have $\mathrm{rank}(\bZ^*) = n-r^*$ and kernel of $\bZ^*$ is equal to the column space of $\bX^*$, i.e., $\mathrm{ker}(\bZ^*) = \mathrm{col}(\bX^*)$. Since $\bX^* = \bs\bs^\T$ and $\bZ^* = \bL-\bA$, we equivalently have $\mathrm{ker}(\bL-\bA) = \mathrm{col}(\bs)$. As $\bZ^*$ is feasible for the dual, then $\bZ^* = \bL-\bA \succeq 0$, and consequently $\inner{\bu, (\bL-\bA)\bu} \geq 0$, for all $\bu\in\reals^{n \times r}$.

Now consider the quadratic form $h(\bu) := \inner{\bu, (\bL-\bA)\bu}$ over $\bu \in T_{\bs}\M_r$. First, we show that $h(\bu) = 0$ if and only if $\bu \in \V_{\bs}$. The {\em only if} direction of the proof is straightforward, i.e., $(\bL-\bA)\bs=0$ and $\bu=\bs \bB$ for some skew-symmetric matrix $\bB$ directly imply $h(\bu)=0$ for all $\bu \in \V_{\bs}$. To show the {\em if} direction, let $\bu \in T_{\bs}\M_r$ such that $h(\bu)=0$, or equivalently $\tr((\bL-\bA)\bu\bu^\T) = 0$. As both $\bL-\bA$ and $\bu\bu^\T$ are positive semidefinite matrices, this implies $(\bL-\bA)\bu=0$. Therefore, columns of $\bu$ are in $\mathrm{ker}(\bL-\bA) = \mathrm{col}(\bs)$, which implies there exists $\bB \in \reals^{r \times r}$ such that $\bu=\bs \bB$. As $\bu \in T_{\bs}\M_r$, then $\inner{\s_i,u_i} = \inner{\s_i,\bB^\T\s_i} = \inner{\s_i \s_i^\T, \bB} = 0$, for all $i\in[n]$. Without loss of generality, assume that the last $r-r^*$ columns of $\bs$ are equal to zero. Then, by dual nondegeneracy of the SDP, the principal submatrices of dimension $r^* \times r^*$ of $\{\s_i \s_i^\T\}_{i=1}^n$ spans $\mathcal{S}^{r^*}$. Consider the decomposition
\begin{equation*}
	\bB = 
	\begin{bmatrix}
	\bB_{11} & \bB_{12} \\
	\bB_{21} & \bB_{22}
	\end{bmatrix},
\end{equation*}
where $\bB_{11} \in \reals^{r^* \times r^*}$ and $\bB_{22} \in \reals^{(r-r^*) \times (r-r^*)}$. Then, the dual nondegeneracy implies that $\bB_{11}$ is a skew-symmetric matrix, i.e., $\bB_{11}^\T=-\bB_{11}$. Furthermore, as the last $r-r^*$ columns of $\bs$ are equal to zero, then $\bu=\bs \bB$ does not depend on $\bB_{21}$ and $\bB_{22}$. Therefore, we can pick $\bB_{21} = -\bB_{12}^\T$ and $\bB_{22}=0$ such that $\bB$ is a skew-symmetric matrix and observe that $\bu\in\V_{\bs}$. 

To conclude the proof, we let $\{\bu^\ell\}_{\ell=1}^{n(r-1)}$ be an orthogonal basis to $T_{\bs}\M_r$ such that $\{\bu^\ell\}_{\ell=1}^{s}$ is a basis for $\V_{\bs}$. Let $\bar{\bQ} \in \reals^{n(r-1) \times n(r-1)}$ such that $\bar{Q}_{ij} = \inner{\bu^i, (\bL-\bA)\bu^j}$. Consider the function $\bar{h} : \reals^{n(r-1)} \to \reals^{n(r-1)}$ such that $\bar{h}(\bv) = \bv^\T \bar{\bQ} \bv$ and observe that $\bar{h}(\mathrm{vec}(\bu)) = h(\bu)$. Let $\bLe = [ \mathrm{vec}(\bu^1), \dots, \mathrm{vec}(\bu^s) ] \in \reals^{n(r-1) \times s}$, then $\bv^\T \bar{\bQ} \bv < 0$ for all $\bv \in \mathrm{ker}(\bL) \setminus \{{\boldsymbol 0}\}$. Then, by Finsler's Lemma, $(\bLe^{\perp})^\T \bar{\bQ} \bLe^{\perp} \succ 0$. Equivalently, there exists $\mu>0$ such that $h(\bu) \geq \mu \normf{\bu}^2$ for all $\bu \in T_{\bs}\M_r \setminus \V_{\bs}$.

\section{An Ascent Lemma for the Second-Order Oracle}
\begin{lemma}\label{lem:so_ascent}
Let $\bu^k\in T_{\bs^k}\M_r$ be a tangent vector of $\bs^k$ such that $\normf{\bu^k}=1$, $\inner{\bu^k,\rgrad f(\bs^k)} \geq 0$, and $\inner{\bu^k, \rhess f(\bs^k)[\bu^k]} \geq \varepsilon/2$. Consider the update rule given by the exponential map $\bs^{k+1} = \Exp_{\bs^k}(\bu^k)$, i.e.,
\begin{equation*}
	\s_i^{k+1} = \s_i^k \cos(\norm{u_i^k}t) + \frac{u_i^k}{\norm{u_i^k}} \sin(\norm{u_i^k}t), \quad \text{ for all } i\in[n],
\end{equation*}
where $t=\frac{\varepsilon}{15\norm{A}_1}$ is the step size. These iterates satisfy the following ascent in the function value:
\begin{equation*}
	f(\bs^{k+1}) - f(\bs^k) \geq \frac{\varepsilon^3}{2700 \norm{\bA}_1^2}.
\end{equation*}
\end{lemma}

\begin{proof}
The second-order oracle returns a vector $\bu^k\in T_{\bs^k}\M_r$ such that $\normf{\bu^k}=1$, $\inner{\bu^k,\rgrad f(\bs^k)} \geq 0$, and $\inner{\bu^k, \rhess f(\bs^k)[\bu^k]} \geq \varepsilon/2$. Then we set the next iterate $\bs^{k+1}$ as follows
\begin{equation}\label{eq:sostep}
	\s_i^{k+1} = \s_i^k \cos(\norm{u_i^k}t) + \frac{u_i^k}{\norm{u_i^k}} \sin(\norm{u_i^k}t),
\end{equation}
where $0<t<1$ can be viewed as step size. The third order Taylor approximation to \eqref{eq:sostep} yields:
\begin{equation}
	\s_i^{k+1} = \s_i^k + t u_i^k - \frac{t^2}{2} \norm{u_i^k}^2 \s_i - \frac{t^3}{6} \norm{u_i^k}^2 u_i^k + O(t^4), \nn
\end{equation}
and using this approximation, we obtain
\begin{equation*}
	g_i^{k+1} = g_i^k + t v_i^k - \frac{t^2}{2} \gt_i^k - \frac{t^3}{6} \vt_i^k + O(t^4) ,
\end{equation*}
where
\begin{equation}\label{eq:definitions}
	v_i^k = \sum_{j \neq i} A_{ij} u_j^k \quad \text{and} \quad \gt_i^k = \sum_{j \neq i} A_{ij} \norm{u_j^k}^2 \s_j^k \quad \text{and} \quad \vt_i^k = \sum_{j \neq i} A_{ij} \norm{u_j^k}^2 u_j^k.
\end{equation}
This yields the following Taylor approximation to $f(\bs^{k+1})$, which can be written as follows
\begin{align*}
	f(\bs^{k+1}) & = \sum_{i=1}^n \inner{\s_i^{k+1}, g_i^{k+1}} ,\\
		& = \sum_{i=1}^n \inner{\s_i^k + t u_i^k - \frac{t^2}{2} \norm{u_i^k}^2 \s_i - \frac{t^3}{6} \norm{u_i^k}^2 u_i^k, \, g_i^k + t v_i^k - \frac{t^2}{2} \gt_i^k - \frac{t^3}{6} \vt_i^k} + O(t^4), \\
		& = \sum_{i=1}^n \inner{\s_i^k,g_i^k} + t \sum_{i=1}^n \left( \inner{\s_i^k,v_i^k} + \inner{u_i^k,g_i^k} \right) + \frac{t^2}{2} \sum_{i=1}^n \left( 2 \inner{u_i^k,v_i^k} - \inner{\s_i^k,\gt_i^k} - \norm{u_i}^2 \inner{\s_i^k,g_i^k} \right) ,\\
			& \hspace{2cm} - \frac{t^3}{6} \sum_{i=1}^n \left( \inner{\s_i^k,\vt_i^k} + 3 \inner{u_i^k,\gt_i^k} + 3 \norm{u_i^k}^2 \inner{\s_i^k,v_i^k} + \norm{u_i^k}^2 \inner{u_i^k,g_i^k} \right) + O(t^4).
\end{align*}
Plugging in the definitions in \eqref{eq:definitions}, we obtain
\begin{align*}
	f(\bs^{k+1}) & = \sum_{i=1}^n \inner{\s_i^k,g_i^k} + t \sum_{i=1}^n 2 \inner{u_i^k,g_i^k} + \frac{t^2}{2} \sum_{i=1}^n 2 \left( \inner{u_i^k,v_i^k} - \norm{u_i}^2 \inner{\s_i^k,g_i^k} \right), \\
			& \hspace{2cm} - \frac{t^3}{6} \sum_{i=1}^n 4 \norm{u_i^k}^2 \left( \inner{\s_i^k,v_i^k} + \inner{u_i^k,g_i^k} \right) + O(t^4).
\end{align*}
Using the definitions of $f(\bs^k)$ and its derivatives, the above equality can be written as follows
\begin{align*}
	f(\bs^{k+1}) & = f(\bs^k) + t \inner{\bu^k, \rgrad f(\bs^k)} + \frac{t^2}{2} \inner{\bu^k, \rhess f(\bs^k)[\bu^k]} \\
			& \hspace{2cm} - \frac{t^3}{6} \sum_{i=1}^n 4 \norm{u_i^k}^2 \left( \inner{\s_i^k,v_i^k} + \inner{u_i^k,g_i^k} \right) + O(t^4).
\end{align*}
Here, our aim is to lower bound the remainder term corresponding to the third and higher order terms. To this end, we upper bound the third order terms using the inequalities $\inner{\s_i^k,v_i^k} \leq \norm{\bA}_1$ and $\inner{u_i^k,g_i^k} \leq \norm{\bA}_1$. This yields
\begin{equation*}
	\sum_{i=1}^n 4 \norm{u_i^k}^2 \left( \inner{\s_i^k,v_i^k} + \inner{u_i^k,g_i^k} \right) \leq 8 \norm{\bA}_1 \sum_{i=1}^n \norm{u_i^k}^2 = 8 \norm{\bA}_1,
\end{equation*}
where the last equality follows since $\normf{\bu^k}=1$. Considering the structure of the higher order terms and using similar bounds, we observe that each individual inner product term in higher order expansions can be upper bounded by $\norm{u_i}^{2\lfloor t/2 \rfloor} \norm{\bA}_1$ and consequently we get the following lower bound:
\begin{equation*}
	f(\bs^{k+1}) \geq f(\bs^k) + t \inner{\bu^k, \rgrad f(\bs^k)} + \frac{t^2}{2} \inner{\bu^k, \rhess f(\bs^k)[\bu^k]} - t^3 \norm{\bA}_1 \sum_{\ell=3}^\infty \frac{2^\ell}{\ell!}.
\end{equation*}
Using $\sum_{\ell=3}^\infty \frac{2^\ell}{\ell!} = e^2-5 \leq 5/2$, we obtain
\begin{equation}\label{eq:taylor3}
	f(\bs^{k+1}) \geq f(\bs^k) + t \inner{\bu^k, \rgrad f(\bs^k)} + \frac{t^2}{2} \inner{\bu^k, \rhess f(\bs^k)[\bu^k]} - \frac{5\norm{\bA}_1}{2} t^3.
\end{equation}
Since we are given that $\inner{\bu^k,\rgrad f(\bs^k)} \geq 0$ and $\inner{\bu^k, \rhess f(\bs^k)[\bu^k]} \geq \varepsilon/2$, \eqref{eq:taylor3} yields
\begin{equation*}
	f(\bs^{k+1}) - f(\bs^k) \geq \frac{\varepsilon}{4} t^2  - \frac{5\norm{\bA}_1}{2} t^3.
\end{equation*}
Choosing $t=\frac{\varepsilon}{15\norm{\bA}_1}$ maximizes the right-hand side of the above inequality and guarantees the following ascent in the function value:
\begin{equation*}
	f(\bs^{k+1}) - f(\bs^k) \geq \frac{\varepsilon^3}{2700 \norm{\bA}_1^2}.
\end{equation*}
\end{proof}

\section{Proof of Theorem \ref{thm:iterationcomp}}
We have already proved that each iteration of BCM yields following improvement (by \eqref{eq:subL2}):
\begin{equation}\label{eq:iter1}
	f(\bs^{k+1}) - f(\bs^k) \geq \frac{\normf{\rgrad f(\bs^k)}^2}{2n\norm{\bA}_1} \geq \frac{\varepsilon^3}{2700n\norm{\bA}_1^2},
\end{equation}
where the last inequality holds since BCM is applied at iteration $k$ only if $\normf{\rgrad f(\bs^k)}^2 \geq \frac{\varepsilon^3}{1350\norm{\bA}_1}$. Similarly, each iteration of the second-order oracle yields the following improvement (by Lemma \ref{lem:so_ascent}):
\begin{equation}\label{eq:iter2}
	f(\bs^{k+1}) - f(\bs^k) \geq \frac{\varepsilon^3}{2700 \norm{\bA}_1^2}.
\end{equation}
Hence, an epoch ($n$ iterations) of BCM yields the same amount of function value improvement as an iteration of the second-order oracle. Let
\begin{equation*}
	f^* = \left( 1-\frac{1}{r-1} \right) \text{SDP}(\bA)
\end{equation*}
denote the desired approximation ratio and consider the approximation gap of the solution $\bs$ with respect to $f^*$ that is given by
\begin{equation}
	h(\bs) = f^* - f(\bs).
\end{equation}
The aim of the algorithm is to find a solution $\bs$ that satisfy $h(\bs) \leq \epsilon$ for some $\epsilon>0$. Consider that the BCM2 algorithm runs $\Kbcm$ epochs of BCM and $\Khess$ iterations of the second-order oracle such that a total of $K = n\Kbcm+\Khess$ iterations are made. Let $\mathcal{G} = \{ 0 \leq k \leq K-1 : \normf{\rgrad f(\bs^k)}^2 \geq \frac{\varepsilon^3}{1350\norm{A}_1} \}$ be the set of iterations at which BCM step is taken and let $\mathcal{H} = \{ 0 \leq k \leq K-1 \} \setminus \mathcal{G}$ be the set of iterations at which a second-order oracle step is taken. Then, the approximation gap decreases at each iteration by the following amount:
\begin{equation}\label{eq:iter3}
	h(\bs^k) - h(\bs^{k+1}) \geq \frac{\varepsilon^3}{2700 \norm{\bA}_1^2} \delta_k,
\end{equation}
where, for notational simplicity, we introduced 
\begin{equation}
	\delta_k = 
		\begin{cases}
			\frac{1}{n}, & \text{if } k\in\mathcal{G}, \\
			1, & \text{if } k\in\mathcal{H}.
		\end{cases}
\end{equation}
By Theorem \ref{thm:andrea}, we are given that any $\varepsilon$-approximate concave point $\bs$ satisfies
\begin{equation}\label{eq:approxconc}
	h(\bs) \leq \frac{n}{2} \varepsilon.
\end{equation}
Hence, the right-hand side of \eqref{eq:iter3} can be lower bounded as follows
\begin{equation}\label{eq:iter4}
	h(\bs^k) - h(\bs^{k+1}) \geq \frac{2 \delta_k}{675 n^3 \norm{\bA}_1^2} \, h^3(\bs^k).
\end{equation}
Considering the reciprocal of the approximation gap, we observe that
\begin{align}
	\frac{1}{h^2(\bs^{k+1})} - \frac{1}{h^2(\bs^k)} & = \frac{\left(h(\bs^k)-h(\bs^{k+1})\right) \left(h(\bs^k)+h(\bs^{k+1})\right)}{h^2(\bs^{k+1}) h^2(\bs^k)}, \nonumber\\
		& \geq \frac{2 \delta_k}{675 n^3 \norm{\bA}_1^2} \frac{h(\bs^k) \left(h(\bs^k)+h(\bs^{k+1})\right)}{h^2(\bs^{k+1})}, \label{eq:iter5}
\end{align}
where the inequality follows by \eqref{eq:iter4}. As the right-hand side of \eqref{eq:iter4} is lower bounded by zero, we have $h(\bs^k) \geq h(\bs^{k+1})$. Thus, we can lower bound the right-hand side of \eqref{eq:iter5} as follows
\begin{align}
	\frac{1}{h^2(\bs^{k+1})} - \frac{1}{h^2(\bs^k)} & \geq \frac{4 \delta_k}{675 n^3 \norm{\bA}_1^2}. \label{eq:iter6}
\end{align}
Summing \eqref{eq:iter6} over $k=0,1,\dots,K-1$, we get
\begin{equation*}
	\frac{1}{h^2(\bs^K)} - \frac{1}{h^2(\bs^0)} \geq \sum_{k=0}^{K-1} \frac{4 \delta_k}{675 n^3 \norm{\bA}_1^2} = \frac{4}{675 n^3 \norm{\bA}_1^2} (\Kbcm+\Khess).
\end{equation*}
Given that $\bs^0$ is not an $\varepsilon$-approximate concave point (or else, there is nothing to prove), we have
\begin{equation}
	\frac{1}{h^2(\bs^K)} \geq \frac{4}{675 n^3 \norm{\bA}_1^2} (\Kbcm+\Khess). \label{eq:iter7}
\end{equation}
Since by \eqref{eq:approxconc}, we know that $\frac{1}{h(\bs)} \geq \frac{2}{n\varepsilon}$ for any $\varepsilon$-approximate concave point, then as soon as
\begin{equation}
	\Kbcm+\Khess \geq \frac{675 n \norm{\bA}_1^2}{\varepsilon^2}
\end{equation}
iterations made, BCM2 is guaranteed to return an $\varepsilon$-approximate concave point, i.e., there exists a solution $\bs^k$ for some $1<k<K$ such that $h(\bs^k) \leq \frac{n}{2}\varepsilon$. Since $\{h(\bs^k)\}_{k\geq0}$ is a nonincreasing sequence (as we have already shown in \eqref{eq:iter4}), then the final iterate of the algorithm $\bs^K$ is guaranteed to satisfy $h(\bs^K) \leq \frac{n}{2}\varepsilon$, i.e., $\bs^K$ is an $\varepsilon$-approximate concave point.

\section{Proof of Theorem \ref{thm:iterationcomp_whp}}

\begin{theorem}[{\cite[Theorem 4.2]{kuczynski1992}}]\label{thm:lanczos}
Let $\bA\in\reals^{n \times n}$ be a positive semidefinite matrix, $b\in\reals^n$ be an arbitrary vector and $\lambda_L^\ell(\bA,b)$ denote the output of the Lanczos algorithm after $\ell$ iterations when applied to find the leading eigenvalue of $\bA$ (denoted by $\lambda_1(\bA)$) with initialization $b$. In particular,
\begin{equation*}
	\lambda_L^\ell(\bA,b) = \max \left\{ \frac{\inner{x,\bA x}}{\inner{x,x}} : 0 \neq x \in \mathrm{span}(b,\dots,\bA^{\ell-1}b) \right\}.
\end{equation*}
Assume that $b$ is uniformly distributed over the set $\{b\in\reals^n : \norm{b}=1\}$ and let $\epsilon\in[0,1)$. Then, the probability that the Lanczos algorithm does not return an $\epsilon$-approximation to the leading eigenvalue of $\bA$ exponentially decreases as follows
\begin{equation*}
	\mathbb{P}\left( \lambda_L^\ell(\bA,b) < (1-\epsilon) \lambda_1(\bA) \right)
	\begin{cases}
	\leq 1.648 \sqrt{n} e^{-\sqrt{\epsilon}(2\ell-1)}, & \text{if } 0<\ell<n(r-1), \\
	= 0, & \text{if } \ell \geq n(r-1).
	\end{cases}
\end{equation*}
\end{theorem}

Since the tangent space $T_{\bs}\M_r$ has dimension $n(r-1)$, then we can define a symmetric matrix (where we drop the notational dependency on $\bs$ for simplicity) $\bH \in\reals^{n(r-1) \times n(r-1)}$ that represents the linear operator $\rhess f(\bs)$ in the basis $\{ \bu^1,\dots,\bu^{n(r-1)} \}$ such that $\mathrm{span}(\bu^1,\dots,\bu^{n(r-1)}) = T_{\bs}\M_r$. In particular, letting $H_{ij} = \inner{\bu^i, \rhess f(\bs)[\bu^j]}$ yields the desired matrix $\bH$ and the Lanczos algorithm is run to find the leading eigenvalue of this matrix. Here, it is important to note that $\bH$ is not a psd matrix, so it is required to shift $\bH$ with a large enough multiple of the identity matrix so that the resulting matrix is guaranteed to be positive semidefinite. In particular, by inspecting the definition of $\rhess f(\bs)$ in \eqref{eq:rhessian}, it is easy to observe that $\norm{\rhess f(\bs)}_{\text{op}} \leq 4\norm{\bA}_1$. Therefore, it is sufficient to run the Lanczos algorithm to find the leading eigenvalue of $\widetilde{\bH} = \bH+4\norm{\bA}_1 \bI$, where $\bI$ denotes the appropriate sized identity matrix. On the other hand, we initialize the Lanczos algorithm with a random vector $\bu$ of unit norm (i.e., $\normf{\bu}=1$) in the tangent space $T_{\bs}\M_r$. Notice that $\bu$ can equivalently be represented as a vector $b\in\reals^{n(r-1)}$ in the basis $\{ \bu^1,\dots,\bu^{n(r-1)} \}$ as $\bu = \sum_{i=1}^{n(r-1)} b_i \bu^i$ such that $\norm{b}=1$. Then, by Theorem \ref{thm:lanczos}, we have
\begin{equation*}
	\mathbb{P}\left( \lambda_L^\ell(\widetilde{\bH},b) < (1-\epsilon) \lambda_1(\widetilde{\bH}) \right) \leq 1.648 \sqrt{n(r-1)} e^{-\sqrt{\epsilon}(2\ell-1)}.
\end{equation*}
Letting $\lambda_1(\bH)$ denote the leading eigenvalue of $\bH$, we run the Lanczos algorithm to obtain a vector $b^*$ such that $\norm{b^*}=1$ and $\inner{b^*, \bH b^*} \geq \lambda_1(\bH)/2$. Thus, we want $\mathbb{P}\left( \lambda_L^\ell(\widetilde{\bH},b) < 4\norm{\bA}_1 + \lambda_1(\bH)/2 \right)$ to be small. Setting $\epsilon^* = \frac{\lambda_1(\bH)}{16\norm{\bA}_1}$, we can observe that
\begin{align*}
	\left( 1-\epsilon^* \right) \lambda_1(\widetilde{\bH}) & = \left( 1-\frac{\lambda_1(\bH)}{16\norm{\bA}_1} \right) \left( 4\norm{\bA}_1 + \lambda_1(\bH) \right), \\
		& = 4\norm{\bA}_1 + \frac{3\lambda_1(\bH)}{4} - \frac{(\lambda_1(\bH))^2}{16\norm{\bA}_1}, \\
		& \geq 4\norm{\bA}_1 + \frac{\lambda_1(\bH)}{2},
\end{align*}
where the inequality follows since $\lambda_1(\bH) \leq 4\norm{\bA}_1$. Consequently, we have
\begin{equation*}
	\mathbb{P}\left( \lambda_L^\ell(\widetilde{\bH},b) < 4\norm{\bA}_1 + \lambda_1(\bH)/2 \right) \leq \mathbb{P}\left( \lambda_L^\ell(\widetilde{\bH},b) < (1-\epsilon^*) \lambda_1(\widetilde{\bH}) \right) \leq 1.648 \sqrt{n(r-1)} e^{-\sqrt{\epsilon^*}(2\ell-1)}.
\end{equation*}
By Theorem \ref{thm:iterationcomp}, we know that the Lanczos method is called at most $\left\lceil 675 n \norm{\bA}_1^2 / \varepsilon^2 \right\rceil$ times to search for an $\varepsilon$-approximate concave point and for any non-desired solution we have $\lambda_1(\bH) \geq \varepsilon$ by the definition of $\varepsilon$-approximate concave point. Then, by using a union bound over all calls to the Lanczos method, we conclude that when the Lanczos method is run for $\ell$ iterations, we have the following guarantee
\begin{align*}
	& \mathbb{P}\left( \text{Algorithm \ref{alg:bcm-soo}+\ref{alg:lanczos} fails to return an $\varepsilon$-approximate concave point} \right) \\
		& \hspace{8cm} \leq \left\lceil \frac{675 n \norm{\bA}_1^2}{\varepsilon^2} \right\rceil 1.648 \sqrt{n(r-1)} e^{-\sqrt{\frac{\varepsilon}{16\norm{\bA}_1}}(2\ell-1)}.
\end{align*}
In order to set this probability to some $\delta\in(0,1)$, we let
\begin{equation*}
	\ell^* = \left\lceil \left( \frac{1}{2} + 2 \sqrt{\frac{\norm{\bA}_1}{\varepsilon}} \right) \log\left( \frac{\left\lceil \frac{675 n \norm{\bA}_1^2}{\varepsilon^2} \right\rceil 1.648 \sqrt{n(r-1)}}{\delta} \right) \right\rceil = \Otilde{ \sqrt{\frac{\norm{\bA}_1}{\varepsilon}} \log\left( \frac{n\sqrt{n(r-1)}}{\delta} \right) },
\end{equation*}
where tilde is used to hide poly-logarithmic factors in $\norm{\bA}_1 / \varepsilon$. Since the Lanczos algorithm is guaranteed to return the leading eigenvalue with probability $1$ in at most $n(r-1)$ iterations, then running each Lanczos subroutine for $\min(\ell^*,n(r-1))$ iterations, it is guaranteed that Algorithm \ref{alg:bcm-soo}+\ref{alg:lanczos} returns an $\varepsilon$-approximate concave point with probability at least $1-\delta$.

%

\end{document}